\documentclass[12pt,article]{amsart}
\usepackage[usenames,dvipsnames]{color} 
\usepackage{amssymb}
\usepackage{graphicx}
\usepackage{amsfonts}
\usepackage{amsthm}
\usepackage{amsmath}
\usepackage{empheq}
\usepackage{tabularx}
\usepackage[usenames,svgnames,dvipsnames]{xcolor}
\usepackage{bbm}

\usepackage{subcaption}
\usepackage{caption}

\usepackage{array}
\usepackage{multirow}
\usepackage{subcaption}
\usepackage{epsfig}
\usepackage{epstopdf}

\usepackage{microtype}
	\usepackage{mathrsfs}
	\usepackage{amssymb}
	\usepackage{amsfonts}
	\usepackage{amsbsy}
	\usepackage{latexsym}
	\usepackage{amssymb,latexsym,amsmath,amsthm}
	\usepackage{xcolor}
	\usepackage{epsfig,psfrag,color}
	\usepackage[ruled,linesnumbered]{algorithm2e}
	\SetKwRepeat{Do}{do}{while}%
	\usepackage{graphicx}
\theoremstyle{plain}
\newtheorem{theorem}{Theorem}[section]
\newtheorem{lemma}[theorem]{Lemma}

\theoremstyle{assumption}

\theoremstyle{definition}

\newtheorem{remark}[theorem]{Remark}

\newcommand{\Sp}{ \hspace{0.05cm} }

\newcommand{\Z}{\mathbb{Z}}
\newcommand{\C}{\mathbb{C}}

\numberwithin{equation}{section}

\begin{document}


\title [Beyond Classical Carleman Linearization 
]
{ 
Beyond 
Carleman Linearization of Nonlinear Dynamical System:  Insights from a  Case Study}

\author{Panpan Chen}
\address{Department of Mathematics, University of Central Florida, Orlando, Florida 32816}
\email{panpan.chen@ucf.edu}

\author{Nader Motee}
\address{Mechanical Engineering and Mechanics, Lehigh University,  Bethlehem, PA 18015}
\email{nam211@lehigh.edu}

\author{Qiyu Sun}
\address{Department of Mathematics, University of Central Florida, Orlando, Florida 32816}
\email{qiyu.sun@ucf.edu}

\dedicatory{In memory of Jean-Pierre Gabardo}. 

\thanks{This work was partially supported by ONR N00014-23-1-2779.}

\begin{abstract}
Nonlinear dynamical systems are widely encountered in various scientific and engineering fields.
 Despite significant advances in theoretical understanding, developing complete and integrated frameworks for analyzing and designing these systems remains challenging, which underscores the importance of efficient linearization methods. In this chapter, we
introduce a general linearization framework with emphasis on Carleman linearization and Carleman-Fourier linearization. A detailed case study on finite-section approximation to the lifted infinite-dimensional dynamical system is provided for the dynamical system with its governing function being a trigonometric polynomial of degree one.

\end{abstract}

\maketitle

\section{Introduction}
\label{introduction.section}

Nonlinear dynamical systems are ubiquitous in disciplines ranging from physics and biology to engineering. Although substantial theoretical progress has been made, holistic mathematical frameworks for the systematic analysis and design of these systems continue to exhibit significant challenges. These methodological gaps highlight the vital importance of enhancing the value of transforming nonlinear systems into linearization frameworks. Over time, 
it developed into a leading method for analyzing nonlinear systems, driven by advancements in theory, 
enhanced computational techniques, and  data explosion
\cite{Amini2022, 
Akiba2023, Belabbas2023, brockett2014early, Brunton2016, Chen2024, forets2017explicit, forets2021reachability, Gonzalez2025, harshana2025, Korda2018, KordaMezic2018, kowalski1991nonlinear,   liu2021efficient, 
moteesun2024, pruekprasert2022moment, 
steeb1980non, surana2024, Vaszary2024,  WangJungersOng2023, Wu2024}.
In this chapter, we consider lifting  
 nonlinear dynamical  systems that are described by 
\begin{equation} \label{nonlineardynamical.def0}
\frac{dx}{dt} =  g(x) \ \ {\rm  with }\ \ x(0)=x_0,
\end{equation}
where  $t\ge 0$ is the time variable,
$x:=x(t)\in {\mathbb C}$ represents the state variable, 
$g$ serves as the governing function, and $x_0\in {\mathbb C}$ acts as the initial.

\medskip

We say that a  map $\Psi$
is a {\bf lifting operator} of the nonlinear dynamical system \eqref{nonlineardynamical.def0} if the lifted system can be described as a linear dynamical system,
\begin{equation}\label{lifting.def1} \frac{d\Psi (x)}{dt} = F\big(\Psi(x)\big),\end{equation}
where $F(y)$ is an affine function about $y$.
For  engineering applications,  the lifting operator $\Psi$ should be  appropriately designed so that the original state variable
$x$  can be recovered from its lifting $\Psi(x)$ by some {\bf projection} $\Phi$,
\begin{equation} \label{lifting.def2} x=\Phi(\Psi(x)).\end{equation}
  Carleman 
linearization, Carleman-Fourier linearization and Koopman linearization are  
well-known lifting schemes,  where the resulting lifted systems are represented by linear operators on some infinite-dimensional linear space.  
In this chapter, we consider  Carleman linearization and Carleman-Fourier linearization of the dynamical system
\eqref{nonlineardynamical.def0}
with the governing function $g$
being a trigonometric polynomial, i.e.,
\begin{equation}\label{complexdynmaic.fdef} 
g(x) =\sum_{m=-M}^M g_m e^{imx}
 \end{equation}
 for some constant Fourier coefficients $g_m, -M\le m\le M$.

 \medskip 

  Carleman linearization
  has lifting and projection operators  defined by  
\begin{equation}\Psi(x)=[x, x^2, \ldots, x^N, \ldots]^T\ \ {\rm and}\ \ 
\Phi[y_1,  y_2, \ldots, y_N, \ldots]^T=y_1.\end{equation}
 It has been embraced by the control system community in a variety of successful applications, 
such as optimal nonlinear control design,  model predictive control, state estimation and feedback control.   These successes underscore the value of Carleman linearization in transforming nonlinear control problems into a linear paradigm, where powerful analytical tools can be applied. This transformation not only facilitates deeper insights into system behavior, but also enables the development of effective control strategies for nonlinear dynamics 
 \cite{amini2020approximate, amini2020quadratization, brockett2014early, forets2021reachability, hashemian2015fast,  krener1974linearization, 
 liu2021efficient, loparo1978estimating, minisini2007carleman, pruekprasert2022moment, rauh2009carleman, rotondo2022towards}.
  The key idea in Carleman linearization is to {\bf decouple}
auxiliary variables $x^n, n\ge 1$, which represent higher-order powers of the state variable $x$. The main challenge is that the state matrix of the lifted infinite-dimensional linear system is the product of a diagonal matrix with {\bf unbounded} diagonal entries and a Laurent matrix with exponential off-diagonal decay, and hence the corresponding infinite-dimensional system is {\bf not amenable} to any of the existing tools for analysis.
Carleman linearization is applicable for dynamical systems with the governing function being analytic around the origin, and it can be viewed as the dynamical system counterpart to the Maclaurin expansion for analytic functions.  Similar to the behavior of the Maclaurin expansion of an analytic function near the origin, a finite-section approximation to the Carleman linearization with higher orders offers significantly improved accuracy and a broader time range of validity compared to the conventional linearized  dynamical model,  especially when
the dynamical system has the origin as its equilibrium and the initial is not far away from the  origin.
In  Section \ref{Carleman.section},  we consider Carleman  linearization  of the  dynamical system
\eqref{nonlineardynamical.def0}
with the trigonometric polynomial governing function $g$ in \eqref{complexdynmaic.fdef},
and we present numerical demonstration for the Carleman linearization of 
the illustrative dynamical system with 
governing function
 \begin{equation}\label{simpleexample2.eq1}
g(x)= a(1-e^{ix})
\end{equation}
for some nonzero $0\ne a\in {\mathbb C}$. 
As the dilated state $x(t/|a|)$
and the reflected state
$-\Re x+i\Im x$
satisfy \eqref{simpleexample2.eq1} with the parameter $a$ replaced by $a/|a|$ and $-\Re a+i \Im a$, respectively,  we may normalize the dynamical system \eqref{simpleexample2.eq1}
so that 
\begin{equation} \label{simpleexample2.eq2a}
 a=e^{i\phi} \ \ {\rm for \ some} \  \ \phi\in [-\pi/2, \pi/2].
\end{equation}
Here and hereafter, we denote the real and imaginary  parts of a complex number $z \in \C$ by $\Re z$ and $\Im z$ respectively.
With the normalization \eqref{simpleexample2.eq2a} on the parameter $a$, one may verify that 
the dynamical system \eqref{nonlineardynamical.def0} with the governing function \eqref{simpleexample2.eq1}
has the origin as an equilibrium  and
 its solution can be  explicitly  expressed as
\begin{equation} \label{simpleexample2.eq2}
    x(t) = at + x_0 + i \ln\left(1 +(e^{ait} - 1)e^{ix_0}\right)
\end{equation}
 in a short time period,  see Appendix \ref{example.section}. Here  we use ${\rm Arg}(z)\in (-\pi, \pi]$ to denote the angle  of a nonzero complex number $z\neq 0$, and set $\ln(z)=\ln|z|+i{\rm Arg}(z)$.

 \medskip

As the  governing function $g$
in \eqref{complexdynmaic.fdef}
 exhibits periodic behavior, a natural  linearization of the dynamical system \eqref{nonlineardynamical.def0} is to adopt the Fourier system 
 $\{e^{ix}, e^{-ix}, e^{2ix}, e^{-2ix}, \ldots\}$,  
 instead of $\{x, x^2, \ldots\}$ 
 in the Carleman linearization \eqref{Carleman.eq}. However, 
    the state matrix of the associated linearization \eqref{CarlemanFourier.eq0}
does not have an upper triangular structure, which is crucial for the convergence
analysis
in the Carleman linearization; see Theorem \ref{Carleman.thm}. This highlights the critical need to reassess the linearization approach for the nonlinear dynamical system  \eqref{nonlineardynamical.def0}.
In Section \ref{CF.section}, we introduce  Carleman-Fourier linearization to the nonlinear dynamical system  \eqref{nonlineardynamical.def0}  with  governing function $g$ in \eqref{complexdynmaic.fdef}.
The Carleman-Fourier linearization in \eqref{CF.eq5}
 is based on the observation that
 the dynamical system associated with the extended state vector ${\bf x}=[x_1, x_2]^T$, {\bf decoupled} from state variables $x_1=x$ and $x_2=-x$,  has its governing field being a trigonometric polynomial with  nonnegative frequencies only.
The proposed Carleman-Fourier linearization  \eqref{CF.eq5}
 has lifting and projection operators defined by
\begin{equation}\Psi(x)= [\exp(i {\pmb \alpha}^T {\bf x})]_{\pmb \alpha\in {\mathbb Z}_{++}^2}
\ \ {\rm and}\ \
 \Phi([y_{{\pmb \alpha}}]_{\pmb \alpha\in {\mathbb Z}_{++}^2}) = -i \ln y_{[1, 0]}, \end{equation}
where ${\bf x}=[x, -x]^T$ and
 $\Z_{++}^2$ is the set of all nonzero pairs of nonnegative integers.
The linearization proposed in \eqref{CF.eq5} inherently preserves the structure of the triangular matrix of the upper block, allowing mathematically tractable and rigorous convergence analysis.
Furthermore, it is  well-adapted to dynamical systems governed by periodic vector functions, leading to a more structured and efficient embedding of periodic nonlinear dynamics within a linear framework; see Section
\ref{finitesectionapproximation.section}.

For the case that the governing function $g$ in \eqref{complexdynmaic.fdef} has a nonnegative frequency only, i.e.,  \eqref{assumption0.nonnegative} holds, 
the Carleman-Fourier linearization in \eqref{CF.eq5} can be reduced to a scaled form for each block. In particular,  we may employ the following lifting and projection operators:
\begin{equation}\Psi(x)= [\exp(ix), \exp(2ix), \ldots,]^T
\ \ {\rm and}\ \
 \Phi([z_1, z_2, \ldots]^T) = -i \ln z_{1}. \end{equation}
 Since both the lifted vector and state matrix in the concise  formulation  \eqref{CF.eq7} are subsets of their counterparts in the Carleman-Fourier linearization described in \eqref{CF.eq5}, we maintain using the terminology of Carleman-Fourier linearization for the resulting system \eqref{CF.eq7}.
 For our case study with the governing function  in  \eqref{simpleexample2.eq1},
we observe that  the state matrix in  the corresponding Carleman-Fourier linearization exhibits an upper-triangular structure with bandwidth one, see \eqref{CF.eq8}.

The state matrix ${\bf B}$ in the proposed Carleman-Fourier linearization \eqref{CF.eq5} does not constitute a bounded operator in
$\ell^2({\mathbb Z}_{++}^2)$ which consists of all square-summable sequences on ${\mathbb Z}_{++}^2$.
  This lack of boundedness prevents the direct application of the standard theories for Hilbert space dynamical systems to analyze the lifted system and, by extension, the original nonlinear dynamical system.
 An alternation to explore the infinite-dimensional lifted system
 \eqref{CF.eq5}
is its finite-section approximation, see
\eqref{CF.eq41}.
In Theorem \ref{maintheoremanalytic.thm1}, we show that
the first block 
in the finite-section approximation \eqref{CF.eq41} converges to the exponential  
of the extended state vector over some time range.
For the governing function in \eqref{simpleexample2.eq1}, we derive an explicit solution of the dynamical system associated with the corresponding finite-section approximation; see \eqref{simpleexample2.eq5-}. As a sequence, the first component in the finite-section approximation is essentially
the Maclaurin polynomial of the exponential of the original state
in \eqref{simpleexample2.eq2},
and hence we have the explicit formula for
its exponential convergence rate and time range; see \eqref{simpleexample2.eq5+}
and \eqref{maximaltimerage.orderone}.

\smallskip

This chapter is organized as follows. 
In Section \ref{Carleman.section}, we present the Carleman linearization framework for the nonlinear dynamical system \eqref{nonlineardynamical.def0} with the governing function $g$ given in \eqref{complexdynmaic.fdef}, see \eqref{Carleman.eq}. Also we numerically demonstrate the exponential convergence of finite-section approximations to
the Carleman linearization  when the governing function $g$ is specified in \eqref{simpleexample2.eq1}, see Theorem \ref{Carleman.thm} and Figures \ref{classicalcarleman} and \ref{classicalcarleman.nonzero}.
In Section \ref{CF.section}, we introduce the Carleman-Fourier linearization with extended variables method, see \eqref{CF.eq5}. This approach removes exponential terms with negative frequencies from the original nonlinear system, and transforms it into an infinite-dimensional linear system whose state matrix has a block-upper triangular structure.  A concise version of Carleman-Fourier linearization is also discussed when the governing function has nonnegative frequencies only, see \eqref{CF.eq7}.
In Section \ref{finitesectionapproximation.section}, we consider the finite-section approximation of the Carleman-Fourier system \eqref{CF.eq5}, providing a computationally tractable finite-dimensional alternative and an effective method for approximating nonlinear dynamical systems. We show that the first block in the finite-section approximation has exponential convergence to 
 the exponential of the state variable in the original dynamical system  over a specified time range.  
To illustrate the Carleman-Fourier linearization, in Section \ref{case.section} we derive an exact time range and convergence rate for the finite-section approximation with the governing function in \eqref{simpleexample2.eq1}. 
Our numerical demonstrations indicate that the Carleman-Fourier linearization outperforms standard Carleman linearization when the imaginary part of the initial state takes a large value,
while Carleman linearization remains superior for systems with initial states near the origin.

\section{Carleman Linearization}
\label{Carleman.section}

In this section, we consider Carleman linearization of the nonlinear dynamical system described in \eqref{nonlineardynamical.def0} with the governing function $g$ given in \eqref{complexdynmaic.fdef}, and  discuss the exponential convergence of the first component
in the finite-section approximation to the state variable of the original dynamical system \eqref{nonlineardynamical.def0}, see
\eqref{Carleman.eq} and Theorem \ref{Carleman.thm}. 
  The reader may refer to \cite{abdia2023, Amini2022, amini2021error, forets2017explicit, kowalski1991nonlinear, liu2021efficient, steeb1980non, Wu2024} and references therein for 
the Carleman linearization of nonlinear dynamical systems with the governing vector fields being analytic around the origin.

The key idea in Carleman linearization is to decouple
auxiliary variables $x^n, n\ge 1$, that represent higher-order powers of the state variable $x$. 
With Maclaurin expansion 
\begin{equation}\label{g.Taylor}
g(x) = \sum_{m=-M}^M g_m \sum_{n=0}^\infty 
\frac{(imx)^n}{n!}=:
\sum_{n=0}^\infty c_n x^n 
\end{equation}
for  the governing trigonometric polynomial  $g$ 
in \eqref{complexdynmaic.fdef}, 
we obtain from \eqref{nonlineardynamical.def0} that
\begin{equation}
\frac{dx^n}{dt}= n x^{n-1} \frac{dx}{dt}=  n x^{n-1} g(x)= n\sum_{n'=n-1}^\infty c_{n'-n+1}
x^{n'}, \ \ n\ge 1. 
\end{equation}
Grouping the above ODEs together yields
  the following infinite-dimensional linear system of ODEs, 
 \begin{equation}\label{Carleman.eq}
\frac{d\mathbf{x}}{dt} = \mathbf{A}\mathbf{x} + \mathbf{a},
\end{equation}
where \(\mathbf{x} = [x, x^{2}, \ldots, x^{N}, \ldots]^{\mathsf{T}}\), \(\mathbf{a} = [c_{0}, 0, \ldots, 0, \ldots]^{\mathsf{T}}\), and the state matrix \(\mathbf{A} = [n c_{n'-n+1}]_{n,n'=1}^{\infty}\) is given by
\begin{equation}\label{Carlemanmatrix.def-A}
\mathbf{A}   = 
\begin{bmatrix}
 c_{1} &  c_{2} & \cdots & c_{N-1} & c_{N} & \cdots \\[6pt]
 2c_{0} & 2c_{1} & \cdots & 2c_{N-2} & 2c_{N-1} & \cdots \\[6pt]
     & 3c_0 & \cdots & 3c_{N-3}& 3c_{N-2} & \cdots \\[6pt]
       &       & \ddots & \vdots & \vdots & \ddots \\[6pt]
        &        &      & N c_{0} & N c_{1} & \ddots \\[6pt]
        &        &        &         & \ddots  & \ddots
\end{bmatrix}.
\end{equation} 
 We call    the  infinite-dimensional dynamical system \eqref {Carleman.eq} obtained from decoupling $x^n, n\ge 1$, as 
 {\bf  Carleman linearization}  of the nonlinear dynamical system
 \eqref{nonlineardynamical.def0}.

The state matrix ${\bf A}$ in the Carleman linearization \eqref{Carleman.eq} is
the product of a diagonal matrix with unbounded entries and a Laurent matrix with exponential off-diagonal decay,
\begin{equation*} 
\mathbf{A}
 = 
\begin{bmatrix}
 1 & & & && &    \\[6pt]
  & 2 & & && &  \\[6pt]
        & &  3 & & & &       \\[6pt]
& & & \ddots & & & \\
&        &      &  & N & &    \\[6pt]
        &        &        &         &  & \ddots
\end{bmatrix}
 \begin{bmatrix}
 c_{1} &  c_{2} & \cdots & c_{N-1} & c_{N} & \cdots \\[6pt]
 c_{0} & c_{1} & \cdots & c_{N-2} & c_{N-1} & \cdots \\[6pt]
     & c_0 & \cdots & c_{N-3}& c_{N-2} & \cdots \\[6pt]
       &       & \ddots & \vdots & \vdots & \ddots \\[6pt]
        &        &      &  c_{0} &  c_{1} & \ddots \\[6pt]
        &        &        &         & \ddots  & \ddots
\end{bmatrix}.
\end{equation*} It is an upper-triangular matrix when
the origin is an equilibrium of the nonlinear dynamical system
 \eqref{nonlineardynamical.def0}, i.e., 
\begin{equation}\label{origin.requirement}
    g(0)=c_0=0.
\end{equation}
 By exploiting the upper triangular structure of the state matrix ${\bf A}$ and building on techniques developed in \cite{Amini2022}, we  may use the following {\bf finite-section approximation} framework
   to  approximate the infinite-dimensional linear dynamical system  \eqref{Carleman.eq},
 \begin{equation}\label{finitesection.Carleman}
 \frac{d {\bf x}_N}{dt}
= {\bf A}_N {\bf x}_N+{\bf a}_N
\end{equation}
 with the initial condition $ {\bf x}_N(0)= [x_0, x_0^2, \ldots, x_0^N]^{\mathsf T}$, 
 where $ {\bf x}_N = [x_{1, N}, \ldots, x_{N, N}]^{\mathsf{T}} $,  ${\bf a}_N=[c_0, \ldots, 0]^T$, and the state matrix 
 $$ {\bf A}_N=
\begin{bmatrix}
 c_{1} &  c_{2} & \cdots & c_{N-1} & c_{N}  \\[6pt]
 2c_{0} & 2c_{1} & \cdots & 2c_{N-2} & 2c_{N-1}  \\[6pt]
     & 3c_0 & \cdots & 3c_{N-3}& 3c_{N-2}  \\[6pt]
       &       & \ddots & \vdots & \vdots \\[6pt]
        &        &      & N c_{0} & N c_{1}
\end{bmatrix}=
\begin{bmatrix}
 1 &     \\
 & 2 &    \\[6pt]
       &        & \ddots &  \\[6pt]
       &        &      & N 
\end{bmatrix}
\begin{bmatrix}
 c_{1} &  c_{2} & \cdots & c_{N-1} & c_{N}  \\[6pt]
 c_{0} & c_{1} & \cdots & c_{N-2} & c_{N-1}  \\[6pt]
     & c_0 & \cdots & c_{N-3}& c_{N-2}  \\[6pt]
       &       & \ddots & \vdots & \vdots \\[6pt]
        &        &      &  c_{0} & c_{1}
\end{bmatrix}
$$
  is the leading $ N \times N $ principal submatrix of the state matrix $ {\bf A} $
  of the Carleman linearization.
  Unlike the infinite-dimensional system setting, 
   we observe that $x_{k,N}, 1\le k\le N$
in the finite-dimensional system above does not {\bf not} satisfy the coupling property $x_{k,N}= (x_{1, N})^k, 1\le k\le N$.  
We also notice that the finite-dimensional dynamical system 
\eqref{finitesection.Carleman}
with $N=1$
is the conventional linearized dynamical model at the origin,
\begin{equation}
\frac{dx_1}{dt}= c_0+c_1 x_1=  g(0)+g'(0) x_1.
    \end{equation}
In this regard, we may view the Carleman linearization and its finite-section approximation
as an {\bf analogue} of the Maclaurin expansion
for an analytic function
and its Macluarin polynomial approximation in the context of dynamical systems.

Stated in the following theorem is that the first component $ {x}_{1,N}(t)$ for $N\ge 1$, in the finite-section approximation 
\eqref{finitesection.Carleman}
 provides an exponential approximation to the solution $x(t)$ of the original dynamical system \eqref{nonlineardynamical.def0}
 in a short time range, provided that
the origin is an equilibrium of the dynamical system, i.e.,  \eqref{origin.requirement} holds.

\begin{theorem}\label{Carleman.thm}
Consider the dynamical systems 
\begin{equation} \label{Carleman.thm.eq0} 
\frac{dx}{dt} =f(x)\end{equation}
with the governing function $f$ being an analytic function  with Maclaurin expansion
$$f(x)=\sum_{n=1}^\infty c_n x^n$$
and Maclaurin coefficients satisfying
\begin{equation}\label{Carleman.thm.eq1}
|c_n|\le C_0 \frac{R_0^{n-1}}{n!}, \ \ n\ge 1,
\end{equation}
for some positive constants $C_0$ and $R_0$,
and let $x_{1, N}$
be the first component of the state vector in the finite-section approximation of order $N$ to its Carleman linearization
\eqref{Carleman.eq}.
Then 
\begin{equation}\label{Carleman.thm.eq2}
|x_{1, N}(t)-x(t)| \le 
\frac{ \tilde R_0 e^{\tilde R_0}}{\sqrt{2\pi} R_0} N^{-3/2}
\Big(\frac{ R_0|x_0| e}{\tilde R_0} e^{C_0  (1+ 1/\tilde R_0) e^{{\tilde R}_0} t } \Big)^N
\end{equation}
holds for all $ 0 \le t \le  T^* $ and $ N \ge 1 $, 
where 
\begin{equation} \label{tildeR0.def}
\tilde R_0=\max(1,  R_0|x_0|e^2)
\end{equation} 
and
\begin{equation}\label{Tstar.def}
T^*=     \frac{\tilde R_0}{C_0(\tilde R_0+1)e^{\tilde R_0} } \min\Big(\ln \frac{\tilde R_0}{eR_0|x_0|}, 2\Big).
\end{equation}
\end{theorem}

We follow the argument in \cite{Amini2022} to prove Theorem \ref{Carleman.thm}. For the completeness of this chapter, we include a sketch proof
in Appendix \ref{proof.section}.

\medskip

 For the  governing function $g$ in  \eqref{simpleexample2.eq1}, 
the corresponding Carleman linearization \eqref{Carleman.eq}
is homogeneous (i.e., ${\bf a}={\bf 0}$) and
has the upper triangular state matrix  
\begin{equation}
{\bf A}= -a \begin{bmatrix}
    i&   \frac{i^2 }{2!}& \cdots  & \frac{i^{N-1}}{(N-1)!} &  \frac{  i^{N}}{N!} & \cdots \\
     & 2i & \cdots  &   \frac{2i^{N-2}}{(N-2)!}   &  \frac{2i^{N-1}}{(N-1)!} & \cdots \\
     &  &\ddots & \vdots &  \vdots &\ddots \\
         &  & & (N-1)i &  \frac{(N-1)i^2}{2!} &\cdots\\
            &  & &  & Ni &  \cdots\\
            &  & &  &  &  \ddots
                  \end{bmatrix}, \end{equation}
and its finite-section approximation 
is given by 
\begin{equation}
\label{simpleexample2.eq7}
\hspace{-0.1cm} \frac{d}{dt} \begin{bmatrix}
    {{x}}_{1,N}(t) \\ {{x}}_{2,N}(t) \\ \vdots\\ \vdots \\ {x}_{N-1,N}(t) \\ x_{N,N}(t)
\end{bmatrix}
   = -a \begin{bmatrix}
    i&   \frac{i^2 }{2!}& \cdots  & \cdots &  \frac{i^{N-1} }{(N-1)!} &  \frac{  i^{N}}{N!} \\
     & 2i & \cdots  & \cdots  &   \frac{2 i^{N-2}}{(N-2)!} &  \frac{2i^{N-1}}{(N-1)!} \\
     &  &\ddots & \ddots & \vdots & \vdots \\
         &  & & \ddots & \vdots & \vdots \\
         &  & & & (N-1)i & \frac{(N-1)i^2}{2!} \\
         &  & & &  & Ni
      \end{bmatrix}
     \begin{bmatrix}
    {x}_{1,N}(t) \\ {x}_{2,N}(t) \\ \vdots \\ \vdots\\ x_{N-1, N}(t) \\{x}_{N,N}(t)
\end{bmatrix} 
\end{equation}
with initial $x_{k, N}(0)=x_0^k$ for  $1\le k\le N$.  Due to the upper triangular property of the state matrix, the above linear system can be solved by a set of scale-valued ODEs: 
$$\frac{d x_{N, N}}{dt}= -Nai x_{N, N}\   \ {\rm with}\  \ x_{N, N}=x_0^N$$
and
$$\frac{dx_{k, N}}{dt}=-kai x_{k, N}-ka\sum_{k'=k+1}^N \frac{i^{k'-k+1}}{(k'-k+1)!} x_{k', N}\ \ {\rm with}\  \ x_{k, N}(0)=x_0^k$$
inductively for $k=N-1, \ldots, 1$.

\begin{figure}[t] 
  \setkeys{Gin}{width=\linewidth}
  \setlength\tabcolsep{2pt}
  \begin{tabularx}{\textwidth}{XXX}
    \includegraphics[width=4.9cm, height=3.8cm]{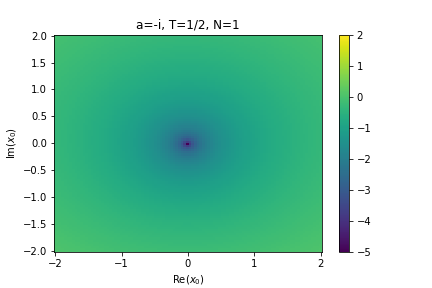} &
    \includegraphics[width=4.9cm, height=3.8cm]{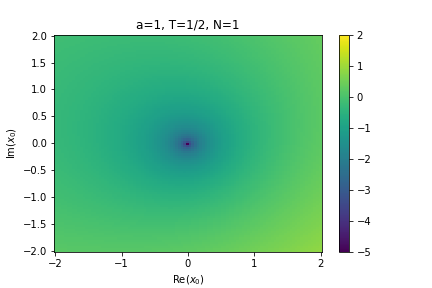} &
    \includegraphics[width=4.9cm, height=3.8cm]{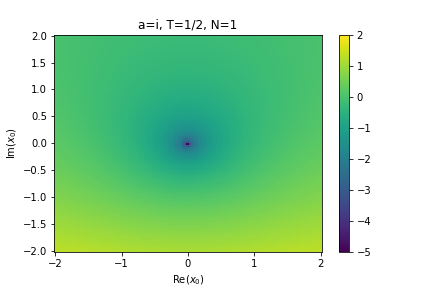}   \\
    \includegraphics[width=4.9cm, height=3.8cm]{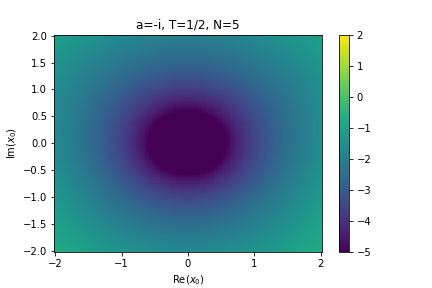} &
    \includegraphics[width=4.9cm, height=3.8cm]{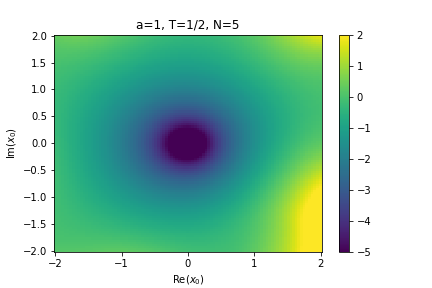} &
    \includegraphics[width=4.9cm, height=3.8cm]{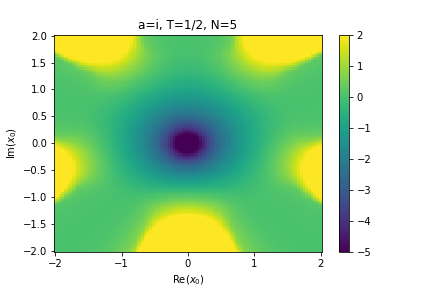}   \\
    \includegraphics[width=4.9cm, height=3.8cm]{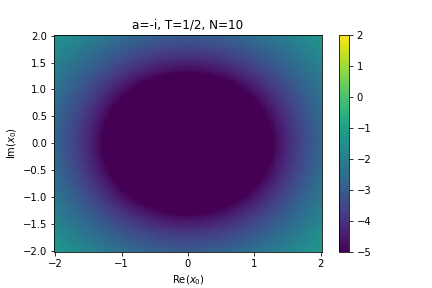} &
    \includegraphics[width=4.9cm, height=3.8cm]{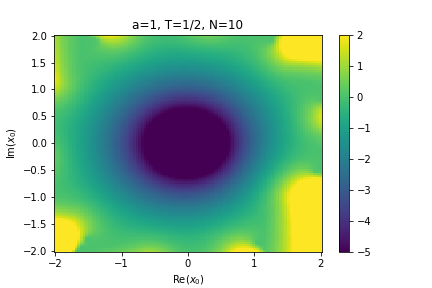} &
    \includegraphics[width=4.9cm, height=3.8cm]{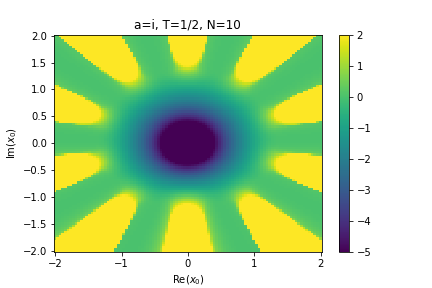}
  \end{tabularx}
\caption{The figures above display the finite-section approximation errors, given by $\max(\min(E_{C}(x_0, T^*, N), 2), -5)$ in \eqref{ETC.def00}, over the domain $-2 \le \Re x_0, \Im x_0 \le 2$, with a
fixed time range $T^*= 1/2$. The columns correspond to value $a = -i$ (left), $1$ (middle), and $i$ (right), while the rows correspond to $N = 1$ (top), $5$ (middle), and $10$ (bottom).  
          }\label{classicalcarleman}
\end{figure}

For the dynamical system with the governing function $g$ given in \eqref{simpleexample2.eq1},
the requirement \eqref{Carleman.thm.eq1} in Theorem \ref{Carleman.thm} is satisfied with $C_0=|a|=1$ and $R_0=1$. Hence as a consequence of  Theorem \ref{Carleman.thm}, \begin{equation}\label{simpleexample.convergence}
|x_{1, N}(t)-x(t)| \le (2\pi)^{-1/2}
 \tilde R_0 e^{\tilde R_0} N^{-3/2}
\Big(\frac{ |x_0| e}{\tilde R_0} e^{e^{{\tilde R}_0} (1+ 1/\tilde R_0) t } \Big)^N
\end{equation}
hold for all $ 0\le t\le
 \frac{\tilde R_0}{(\tilde R_0+1)e^{\tilde R_0} } \min\big(\ln \frac{\tilde R_0}{|x_0|e}, 2\big)$, 
where $x_{1, N}$ is the first component 
of the state vector in \eqref{simpleexample2.eq7},
and $\tilde R_0=\max(1, |x_0| e^2)$.
Shown in  Figure \ref{classicalcarleman} is
the  approximation error measurement
\begin{equation}\label{ETC.def00}
E_C(x_0, T^*, N)= \max_{0\le t\le T^*} \Sp \log_{10} \big|e^{i(x_{1, N}(t)-x(t))}-1\big|,
\end{equation}
where the initial $x_0$ is selected in the domain $[-2, 2]+[-2, 2]i$,  and $a=-i, 1, i$
from left to right,  and $N=1, 5, 10$ from top to bottom.
A comparison of figures across rows reveals that, as expected, 
the finite-section approximation achieves higher accuracy when the initial value is close to the origin (where the governing function is well-approximated by polynomials of low degrees), and the approximation error further decreases as the truncation order $N$ increases, provided that the finite-section approximation converges.
 On the other hand,  the first component of the finite-section approximation fails to approximate the original state vector for initial values far from the origin.
 Comparing column-wise results,  we observe that the finite-section approximation performs
 more effectively for $a=-i$ in which 
 where the origin is a stable equilibrium
 than 
for $a=i$ where the origin is an unstable equilibrium.

\begin{figure}[t] 
  \setkeys{Gin}{width=\linewidth}
  \setlength\tabcolsep{2pt}
  \begin{tabularx}{\textwidth}{XXX}
    \includegraphics[width=4.9cm, height=3.8cm]{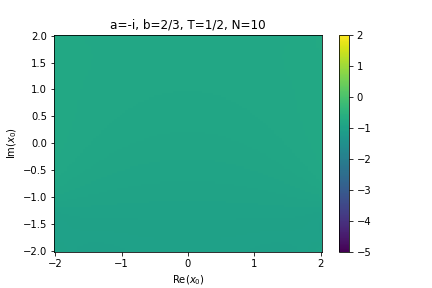} &
    \includegraphics[width=4.9cm, height=3.8cm]{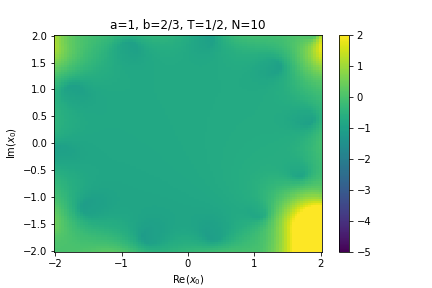} &
    \includegraphics[width=4.9cm, height=3.8cm]{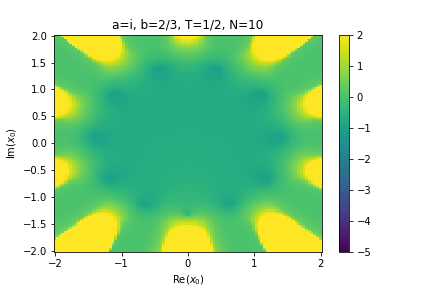}   \\
    \includegraphics[width=4.9cm, height=3.8cm]{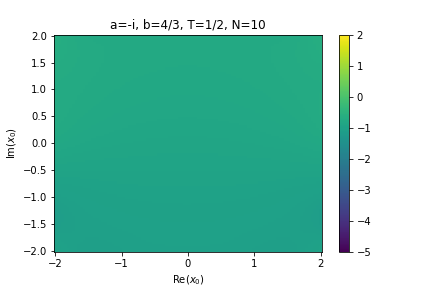} &
    \includegraphics[width=4.9cm, height=3.8cm]{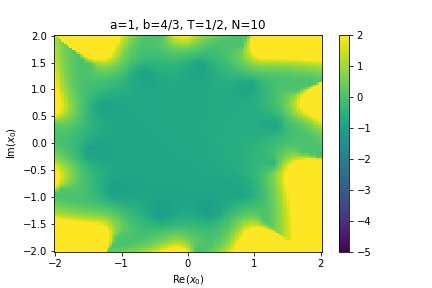} &
    \includegraphics[width=4.9cm, height=3.8cm]{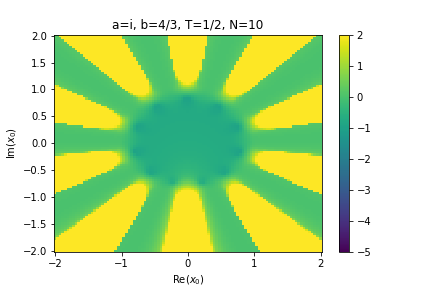}
  \end{tabularx}
\caption{The finite-section approximation errors, given by $\max(\min(E_{C}(x_0, T^*, N), 2), -5)$ in \eqref{ETC.def00}, over the domain $-2 \le \Re x_0, \Im x_0 \le 2$, with a
fixed time range $T^*= 1/2$ and truncation order $N=10$. The columns correspond to value $a = -i$ (left), $1$ (middle), and $i$ (right), while the rows correspond to $b=2/3$ (top) and $b=4/3$ (bottom).  
          }
          \label{classicalcarleman.nonzero}
\end{figure}

Consider the dynamical system
\begin{equation} \label{addynamical}
   \frac{dx}{dt} = a(1-b e^{ix}), \ t\ge 0
\end{equation}
with $0\ne a, b \in {\mathbb C}$, which reduces to the dynamical system \eqref{simpleexample2.eq1} when $b=1$ is selected. The dynamical system \eqref{addynamical} has the origin as its equilibrium if and only if $b=1$.  Shown in Figure \ref{classicalcarleman.nonzero} is the approximation error of the finite-section approach to its Carleman linearization with $b=2/3$ and $4/3$. 
Compared with the numerical results shown in Figure \ref{classicalcarleman} where the origin is an equilibrium point, the finite-section approximation shown in Figure \ref{classicalcarleman.nonzero}  exhibits slower convergence when the initial condition $x_0$ lies near the origin; however, it still achieves exponential convergence over a finite time interval, even though the origin is not an equilibrium point. We conjecture that such exponential convergence property could be mathematically proved 
 for the finite-section approximation to the Carleman linearization of the dynamical system \eqref{nonlineardynamical.def0} under the assumption that the parameter $g(0)$ is close to zero (where the state matrix ${\bf A}$ in
\eqref{Carlemanmatrix.def-A} is no longer an upper triangular matrix).
The main challenge is that the conventional perturbation approach is not directly applicable due to the presence of an unbounded operator in the infinite-dimensional lifted system

\section{Carleman-Fourier Linearization}
\label{CF.section}

 In this section,  we consider Carleman-Fourier linearization  for the nonlinear dynamical system described in \eqref{nonlineardynamical.def0} with the governing function $g$ given in \eqref{complexdynmaic.fdef}.
The extended Carleman-Fourier linearization framework for nonlinear dynamical systems with governing fields having multiple fundamental frequencies can be found in \cite{Chen2024, moteesun2024}.

 Let $\mathbb{Z}_0$ be the set of all nonzero integers.    
The classical Carleman linearization approach, based on monomial functions $x^n, n\ge 1$, 
can be naturally extended to the  dynamical system
\eqref{nonlineardynamical.def0}
with periodic governing function 
$g$  in  \eqref{complexdynmaic.fdef},
by considering  the derivatives of the Fourier basis $e^{inx}, n\in {\mathbb Z}_0$.
    By \eqref{nonlineardynamical.def0} and \eqref{complexdynmaic.fdef}, we have 
$$\frac{d e^{inx}}{dt}= in e^{inx} \frac{dx}{dt}=
in  \sum_{m=-M}^M g_m e^{i(n+m)x}, \ n\in {\mathbb Z}_0.$$
Combining the above ODEs together yields
the following linearization of 
the nonlinear dynamical system \eqref{nonlineardynamical.def0},
\begin{equation}\label{CarlemanFourier.eq0}
\frac{d{\bf w}}{dt} = {\bf H}\,{\bf w} + {\bf h},
\end{equation}
where  
$
{\bf w} = [e^{inx}]_{n \in \mathbb{Z}_0}$ is the state vector, the state matrix \({\bf H} = [in\,g_{n'-n}]_{n, n' \in \mathbb{Z}_0}\) and
nonhomogenous term
\({\bf h} = [\,in\,g_{-n}\,]_{n \in \mathbb{Z}_0}\)
are given by
\begin{equation}\label{Carlemanmatrix.def}
{\bf H}=
 =
\begin{bmatrix}
\ddots & \vdots & \vdots & \vdots & \vdots & \ddots\\
\cdots & 2ig_0 & 2ig_{-1} & 2ig_{-3} & 2ig_{-4} & \cdots\\
\cdots & 2ig_0 & 2ig_{-1} & 2ig_{-3} & 2ig_{-4} & \cdots\\
\cdots & ig_1 & ig_0 & ig_{-2} & ig_{-3} & \cdots\\
\cdots & -ig_3 & -ig_2 & -ig_0 & -ig_{-1} & \cdots\\
\cdots & -2ig_4 & -2ig_3 & -2ig_1 & -2ig_0 & \cdots\\
\ddots & \vdots & \vdots & \vdots & \vdots & \ddots
\end{bmatrix},
\end{equation}
and
\[
{\bf h} = [\,\cdots,\; 2ig_{-2},\; ig_{-1},\; -ig_1,\; -2ig_2,\;\cdots]^{\mathsf{T}}
\]
respectively. 
The above linearization 
 captures the periodicity of the governing function $g$ in \eqref{complexdynmaic.fdef}.
However, unlike in the Carleman linearization, the state matrix ${\bf H}$ in the current linearization approach
 does not maintain an upper triangular form, which is crucial to our convergence analysis in 
 Theorem \ref{Carleman.thm}.  
 This fundamental difference highlights the need to reconsider the linearization strategy for the dynamical system  \eqref{nonlineardynamical.def0}
 with the governing function $g$ given in \eqref{complexdynmaic.fdef}.

Our novel approach in \cite{Chen2024, moteesun2024}
is to introduce
an extended state vector \({\bf x} = [x_1, x_2]^{\mathsf{T}}\), where \(x_1 = x\) and \(x_2 = -x\).   
By \eqref{nonlineardynamical.def0} and \eqref{complexdynmaic.fdef}, one may verify that the dynamical system associated with the extended state vector ${\bf x}=[x_1, x_2]^T$ has the governing field being a trigonometric polynomial with  nonnegative frequencies only, 
\begin{equation} \label{CF.eq1}
\frac{d}{dt}\left[\begin{array}{c} x_1\\
 x_2\end{array}\right]=\sum_{m=0}^M \left[\begin{array}{c}  g_{m}\\
-g_{m} \end{array}\right]
e^{imx_1}+\sum_{m=1}^M \left[\begin{array}{c}  g_{-m}\\
-g_{-m}\end{array}\right] e^{imx_2}.
\end{equation}

Set  $y_{\alpha}=e^{i(\alpha_1 x_1+\alpha_2x_2)}$ for 
 $\alpha=[\alpha_1, \alpha_2]\in {\mathbb Z}_+^2$.
By  \eqref{CF.eq1}, we have 
\begin{eqnarray}\label{CF.eq2}
\frac{d y_\alpha}{dt} & \hskip-0.08in = & \hskip-0.08in i y_\alpha \Big (\alpha_1 \frac{d x_1}{dt}+\alpha_2 \frac{d x_2}{dt}\Big)  \nonumber\\
& \hskip-0.08in = &  \hskip-0.08in i  (\alpha_1-\alpha_2)   \Big(\sum_{m=0}^M 
g_m e^{imx_1}+\sum_{m=1}^M g_{-m} e^{imx_2}\Big) y_\alpha\nonumber\\
& \hskip-0.08in = & \hskip-0.08in  i (\alpha_1-\alpha_2) \sum_{\beta\in {\mathbb Z}_+^2} h_{\beta-\alpha} y_{\beta},
\end{eqnarray}
where for $\gamma=[\gamma_1, \gamma_2]^T\in {\mathbb Z}^2$ we define
\begin{equation}\label{CF.eq3}
h_\gamma=\left\{\begin{array}{ll}
g_{\gamma_1}     & {\rm if} \  0\le \gamma_1\le M\ {\rm and}\ \gamma_2=0, \\
g_{-\gamma_2}  & {\rm if} \ \gamma_1=0 \ {\rm and} \ 1\le \gamma_2\le M,     \\
0  & {\rm otherwise}.
\end{array}\right.\end{equation}

For $\alpha=[\alpha_1, \alpha_2]^T\in {\mathbb Z}^2$, denote its order by $|\alpha|=|\alpha_1|+|\alpha_2|$.
Regrouping all equations in \eqref{CF.eq2} with  $|\alpha|=k$ yields
\begin{equation}\label{CF.eq4}
\frac{d {\bf y}_k}{dt}=\sum_{l=k}^{k+M}
{\bf B}_{kl} {\bf y}_l
\end{equation}
with initial ${\bf y}_k(0)={\bf y}_k^0=[e^{ikx_0}, e^{i(k-2)x_0}, \ldots, e^{-i(k-2)x_0}, e^{-ikx_0}]^{\mathsf{T}}, \ k\ge 1$,
where
  ${\bf y}_k=[y_\alpha]_{|\alpha|=k}=[e^{ikx_1}, e^{i((k-1)x_1+x_2)}, \ldots, e^{i(x_1+(k-1)x_2)}, e^{ikx_2}]^{\mathsf{T}} $, 
and
\begin{equation}
  \label{bkl.def}  
{\bf B}_{kl}=[ i(\alpha_1-\alpha_2) h_{\beta-\alpha}]_{|\alpha|=k, |\beta|=l},\  1\le k\le l, \end{equation} are block matrices depending on Fourier coefficients $g_{\pm (l-k)}$.
The dynamical systems \eqref{CF.eq4}
can be rewritten as the following matrix formation:
\begin{equation}\label{CF.eq5}
\frac{d{\bf y}}{dt}
:= \frac{d}{dt}\begin{bmatrix}
     {\bf y}_1\\  {\bf y}_2 \\ \vdots \\ 
     {\bf y}_N \\
  \vdots
\end{bmatrix}
= 
\begin{bmatrix} {\bf B}_{11} & {\bf B}_{12} & \cdots & {\bf B}_{1N} & \cdots\\
& {\bf B}_{22} & \cdots & {\bf B}_{2N}& \cdots\\
& & \ddots & \vdots & \ddots\\
& & & {\bf B}_{NN} & \ddots\\
& & & & \ddots
\end{bmatrix}
\begin{bmatrix}
    {\bf y}_1\\  {\bf y}_2 \\ \vdots \\ 
     {\bf y}_N \\
  \vdots
\end{bmatrix}=:{\bf B} {\bf y}
\end{equation}
This revised  linearization retains the fundamental benefits of the traditional Carleman techniques, such as a block-upper triangular state matrix, circumvents the structural shortcomings 
associated with the direct Fourier-based linearization in \eqref{CarlemanFourier.eq0}. 
 We call the  infinite-dimensional  dynamical system \eqref{CF.eq5} as {\bf Carleman-Fourier linearization} of the  finite-dimensional nonlinear dynamical system \eqref{nonlineardynamical.def0} when
the periodic vector field ${\bf g}$
 satisfies  \eqref{complexdynmaic.fdef}.

For the case that the governing function $g$ has nonnegative frequecies only, i.e.,
\begin{equation}\label{assumption0.nonnegative}
g(x)=\sum_{m=0}^M g_m e^{imx},
\end{equation}
we have
\begin{equation}\label{complexdynamic.def2}
\frac{dz_k}{dt}= ik \sum_{l=k}^{M+k} g_{l-k} z_l
\end{equation}
where $z_k=y_{[k,0]}=e^{ikx_1}$ is the first component of  ${\bf y}_k, k\ge 1$ in \eqref{CF.eq5}.
Then we can reformulate  ODEs in \eqref{complexdynamic.def2} in the following matrix form,
\begin{equation}\label{CF.eq7}
    \frac{d{\bf z}}{dt}={\bf G} {\bf z}, \ t\ge 0,
\end{equation}
 with initial condition ${\bf z}(0)=[ e^{ikx_0} ]_{k=1}^\infty$, where
  $  {\bf z}=[ z_1,  z_2, \ldots,  z_N, \ldots]^{\mathsf{T}}$ 
and
\begin{equation} \label{Carleman.eq6}
{\bf G}=
\begin{bmatrix}
    ig_0 &  ig_1 & \dots & ig_{N-1}  & \cdots\\
      &2ig_0 & \cdots &  2ig_{N-2}  & \cdots\\
      & & \ddots & \vdots & \ddots\\
      & & & Ni g_0 & \cdots \\
      & & & & \ddots
\end{bmatrix}.
\end{equation}
As the state vector ${\bf z}$
and the state matrix ${\bf G}$ in \eqref{CF.eq7}
are parts of the state vector ${\bf y}$
and the state matrix ${\bf B}$ in 
the infinite-dimensional dynamical system 
\eqref{CF.eq5} respectively,  we call  the infinite-dimensional linear system
\eqref{CF.eq7}
as the {\bf  Carleman-Fourier linearization} of the dynamical system  \eqref{nonlineardynamical.def0}
when  the governing function $g$ has nonnegative frequecies only.

For the  governing function in 
\eqref{simpleexample2.eq1}, one may verify that the corresponding Carleman-Fourier linearization is given by
\begin{equation} \label{CF.eq8}
\frac{d}{dt}\begin{bmatrix}
    z_1 \\
 z_2
\\
 z_3\\
 \vdots \\
 z_{N-1} \\
 z_N \\
     \vdots
\end{bmatrix}
= ai
\begin{bmatrix}
    1 &  -1 & 0 & \dots & 0 & 0   & \cdots\\
      &2 & -2 & \cdots &  0  & 0  & \cdots\\
      & &3 & \ddots &    0  & 0  & \cdots\\
      & & & \ddots & \vdots & \vdots & \cdots\\
       & & & & N-1 & -N+1 & \cdots \\
      & & & & & N & \cdots \\
      & & & & & \ddots & \ddots
\end{bmatrix}
\begin{bmatrix}
     z_1 \\
 z_2
\\
 z_3\\
 \vdots \\
 z_{N-1} \\
 z_N \\
     \vdots
\end{bmatrix},
\end{equation}
where the state matrix is  an upper-triangular structure with bandwidth one.

\section{Finite-section Approximation to
Carleman-Fourier linearization}
\label{finitesectionapproximation.section}

For the Carleman-Fourier linearization  \eqref{CF.eq5}, we define its 
{\bf finite-section approximation} of order $N\ge 1$
 by 
\begin{equation}
\label{CF.eq41}
\frac{d}{dt}
\begin{bmatrix}
    {{\bf y}}_{1,N} \\
    {{\bf y}}_{2,N} \\
    \vdots \\
    {{\bf y}}_{N,N}
\end{bmatrix}
=
\begin{bmatrix} {\bf B}_{11} & {\bf B}_{12} & \cdots & {\bf B}_{1N} \\
& {\bf B}_{22} & \cdots & {\bf B}_{2N}\\
& & \ddots & \vdots \\
& & & {\bf B}_{NN}
\end{bmatrix}
\begin{bmatrix}
    {{\bf y}}_{1,N} \\
    {{\bf y}}_{2,N} \\
     \vdots \\
    {{\bf y}}_{N,N}
\end{bmatrix}\ \  {\rm and}\  {\bf y}_{k,N}(0)=
\begin{bmatrix}
e^{ikx_0}\\ e^{i(k-2)x_0}\\ \vdots\\ e^{-ikx_0}
\end{bmatrix}, 1\le k\le N.
\end{equation}
By 
 \eqref{CF.eq3}, \eqref{bkl.def} and \eqref{CF.eq5},
the  finite-section approximation 
\eqref{CF.eq41} of order $N=1$
reduces to 
$$
\frac{d {\bf y}_{1, 1}}{dt} 
= \begin{bmatrix} ig_0 & 0 \\
0 & -ig_0
\end{bmatrix} {\bf y}_{1, 1}
\ \ {\rm and} \ \ {\bf y}_{1, 1}(0)=
\begin{bmatrix} e^{i x_0} \\
e^{-ix_0}\end{bmatrix},
$$
which has the solution 
\begin{equation} \label{CF.eq41a}
  {\bf y}_{1, 1}(t)= [e^{i (x_0+g_0 t)}, e^{-i (x_0+g_0 t)}]^T.
\end{equation}
Similarly 
the  finite-section approximation 
\eqref{CF.eq41} of order  $N=2$
can be explicitly written as
$$
\frac{d}{dt} 
\begin{bmatrix} {\bf y}_{1, 2} \\
{\bf y}_{2, 2}
\end{bmatrix} 
= \begin{bmatrix} ig_0 & 0  & i g_1 & ig_{-1} & 0 \\
0 & -ig_0 & 0 & -ig_1 & -i g_{-1} \\
0 & 0 & 2ig_0 & 0 & 0 \\
0 &0 & 0 & 0 & 0\\
0 & 0 & 0 & 0 & -2ig_0
\end{bmatrix}
\begin{bmatrix} {\bf y}_{1, 2} \\
{\bf y}_{2, 2}
\end{bmatrix} 
\ \ {\rm and} \ \ 
\begin{bmatrix} {\bf y}_{1, 2}(0) \\
{\bf y}_{2, 2}(0)
\end{bmatrix} =
\begin{bmatrix} e^{i x_0} \\
e^{-ix_0}\\
e^{2i x_0}\\
1\\
e^{-2i x_0}
\end{bmatrix},
$$
which has the  solution  given
by
\begin{equation}  \label{CF.eq41b}
\begin{bmatrix} {\bf y}_{1, 2}(t)\\
{\bf y}_{2, 2}(t)
\end{bmatrix}
=
\begin{bmatrix}
(g_1 e^{i(2x_0+g_0t)}+g_{-1}) (e^{ig_0 t}-1)/g_0+ e^{i(x_0+g_0t)}\\ 
(g_1 e^{-i(2x_0+g_0t)}+g_{-1}) (e^{-ig_0 t}-1)/g_0+ e^{-i(x_0+g_0t)}\\ 
  e^{2i (x_0+g_0 t)}\\ 1\\ 
 e^{-2i (x_0+g_0 t)}
\end{bmatrix}.
\end{equation}

The state matrix ${\bf B}$ in the infinite-dimensional system \eqref{CF.eq5} fails to be a bounded operator on the sequence space 
$\ell^2({\mathbb Z}_{++}^2)$, rendering standard Hilbert space dynamical system theory inapplicable.  However,  as shown in the following theorem, 
the first block ${\bf y}_{1, N}$
in \eqref{CF.eq41} still
converges to the exponential $e^{i{\bf x}}$ of the extended
state vector ${\bf x}=[x, -x]^T$ on a time range.

\begin{theorem} {\rm \cite{Chen2024}}
\label{maintheoremanalytic.thm1}
Consider the dynamical system  
\eqref{nonlineardynamical.def0}
 governed by the periodic function $g$ in \eqref{complexdynmaic.fdef}, and 
the finite-section approximation \eqref{CF.eq41}
to its Carleman-Fourier linearization  
\eqref{CF.eq5}.
Write ${\bf y}_{1, N}=[y_{1, N}^+, y_{1, N}^-]^T$
 take $R>e$,  and select  
 the initial $x_0$ of the  dynamical system  
\eqref{nonlineardynamical.def0} so that 
\begin{equation} \label{maintheoremanalytic.thm1.eq1}
|\Im {x}_0|<\ln R-1.
\end{equation} 
Then 
\begin{equation} \label{maintheoremanalytic.thm1.eq2}
 \hspace{0.05cm} \big| {y}_{1, N}^{\pm} (t) \hspace{0.05cm} e^{\mp i x(t)}-1\big| \leq C_0 N^{-3/2} 
 e^{ D_0(R)  N t}
 \Big( \frac{\exp (|\Im {x}_0|+1)}{R }\Big)^{ (e-1)N/(2e-1)},  \ 0\le t\le T_{CF}^*,
\end{equation}
where   $D_0(R)= 2\max \big(|g_0|, 
(|g_1|+|g_{-1}|)R, \cdots,  (|g_M|+|g_{-M}|)R^M \big)$,
\begin{equation}\label{maintheoremanalytic.thm1.eq3}
C_0 = \frac{1}{\sqrt{2\pi} (e-1)} \exp \left( \frac{3e-1}{2e-1} 
|\Im {x}_0|+ \frac{3e-1}{2e-1} \Sp \ln R- \frac{e}{2e-1} \right),
\end{equation}
and 
 \begin{equation} \label{maintheoremanalytic.thm1.eq4}
    T^*_{CF} = \frac{e-1}{(2e-1)D_0(R) } \ln 
    \frac{R}{\exp(|\Im x_0|  + 1)}>0.
\end{equation}
\end{theorem}


Take $T^{**}\le T_{CF}^*$, and select  a sufficiently large order \(N\) in the finite-section approximation \eqref{CF.eq41} such that
\begin{equation}\label{CF.eq44}
  C_0 N^{-3/2} 
 e^{ D_0(R)  N T^{**}}
 \Big( \frac{\exp (|\Im {x}_0|+1)}{R }\Big)^{ (e-1)N/(2e-1)}
 \le \frac{1}{2}.
 \end{equation}
 Then
$     \xi_{1, N}^\pm (t)= -i \ln y_{1, N}^\pm(t), \ 0\le t\le T^{**} $, with the initial
$\xi_{1, N}^\pm(0)=\pm x_0$ is well defined by \eqref{maintheoremanalytic.thm1.eq2} and \eqref{CF.eq44}.
 This together with the observation 
 $$ |z \ {\rm mod} \ 2\pi|\le  4 \epsilon \ \  {\rm for \ all } \  z\in {\mathbb C}  \ {\rm with}\  \ |e^{iz}-1|\le \epsilon\le 1/2, $$
yields  an exponential approximation to the state variable $x$ for the dynamical system \eqref{nonlineardynamical.def0} by $\xi_{1, N}^\pm $ on the time range  $[0, T^{**}]$, 
\begin{eqnarray} \label{maintheoremanalytic.cor1.eq2}
 & & \max\big(|\xi_{1, N}^+ (t)-x(t)|,\ |\xi_{1, N}^- (t)+x(t)|\big) \nonumber\\
 & \hskip-0.08in  \leq & \hskip-0.08in   4 
 C_0 N^{-3/2} 
 e^{ D_0(R)  N t}
 \Big( \frac{\exp (|\Im {x}_0|+1)}{R }\Big)^{ (e-1)N/(2e-1)}, \ \  0 \leq t \leq T^{**}.
\end{eqnarray}

\medskip

 For the Carleman-Fourier linearization \eqref{CF.eq7}  of dynamical systems with governing function having nonnegative frequecies,  we  define its finite-section approximation of order $N\ge 1$ by 
 \begin{equation} \label{CF.eq45}
\frac{d}{dt}\begin{bmatrix}
    z_{1, N} \\
 z_{2, N}
\\
 \vdots \\
 z_{N, N} 
\end{bmatrix}
=
\begin{bmatrix}
    ig_0 &  ig_1 & \dots & ig_{N-1} \\
      &2ig_0 & \cdots &  2ig_{N-2}  \\
      & & \ddots & \vdots \\
      & & & Ni g_0  \\
\end{bmatrix}
\begin{bmatrix}
    z_{1, N} \\
 z_{2, N}
\\
 \vdots \\
 z_{N, N} 
\end{bmatrix} \ \ {\rm and}
\ \ 
\begin{bmatrix}
    z_{1, N}(0) \\
 z_{2, N}(0)
\\
 \vdots \\
 z_{N, N} (0)
\end{bmatrix}=
\begin{bmatrix}
    e^{ix_0} \\
 e^{2ix_0}
\\
 \vdots \\
 e^{Nix_0}
\end{bmatrix}.
\end{equation}
The finite-section approximation \eqref{CF.eq45} of order $N\ge 1$
has solution
$z_{1, 1}(t)=e^{i(x_0+g_0t)}$
for $N=1$, and
$z_{1,2}(t)= g_1 e^{i(2x_0+g_0t)} (e^{ig_0t}-1)/g_0+e^{i(x_0+g_0t)}$
for $N=2$, cf. \eqref{CF.eq41a} and \eqref{CF.eq41b}.
As the state vector and matrix
in the dynamical system \eqref{CF.eq45}
are subvector and submatrix of the finite-section approximation  \eqref{CF.eq41} respectively, we obtain from Theorem \ref{maintheoremanalytic.thm1}
that  $z_{1, N}$
converges to the exponential $e^{ix}$ of the state vector
$x$ of the original dynamical system in a time range. 

\begin{remark}\label{global.remark}
We remark that for the case that the governing function $g$ has nonnegative frequencies only, the first component $z_{1, N}$  in \eqref{CF.eq45}
may converge to the exponential $e^{ix}$ in the entire time range $[0, \infty)$. In particular, it is shown in
\cite{Chen2024} that the above exponential convergence result for the finite-section approximation  \eqref{CF.eq45}
has  exponential convergence rate
\begin{equation} \label{ConvergencerateCarlemanFourierWhole}
\tilde r_{CF}= \frac{(D_0+\mu_0)\|\exp(i{\bf x}_0)\|_2}{\mu_0 R}<1,
\end{equation}
 under the assumption that
 the zero-th Fourier coefficient $g_0$
 and the initial $x_0$ satisfies
   \begin{equation}
\label{assumption2}  \mu_0:=\Im  g_{0} >0
\ \ {\rm and} \ \ \exp(-\Im x_0)< \frac{\mu_0 R} {D_0+\mu_0},
    \end{equation}
where $R> \exp(|\Im x_0|+1)$ and $D_0(R)=\max (|g_0|, |g_1|R, \ldots, |g_M| R^M)$, c.f. \eqref{timerange.orderoneeq1}.
Under the  assumption \eqref{assumption2}   on the governing function $g$ and the initial state $x_0$,
 imaginary part of  the state vector $x(t)$ converges to positive infinity as $t\to \infty$
 and the dynamical system associated with the new state vector $e^{ix}$ (the exponential of the original state vector $x$) is stable.
\end{remark}

For the case that the governing function $g$ is given in \eqref{simpleexample2.eq1}, 
the corresponding finite-section approximation is given by
\begin{equation} \label{CF.eq46}
\frac{d}{dt}\begin{bmatrix}
    z_{1, N} \\
 z_{2, N}
\\
 \vdots \\
 z_{N-1, N}\\
 z_{N, N} 
\end{bmatrix}
= a i
\begin{bmatrix}
    1 &  -1 & \dots & 0 & 0 \\
      &2 & \cdots &  0 & 0   \\
      & & \ddots & \vdots & \vdots \\
      & & & N-1   & -N+1 \\
      & &  & & N   \\
\end{bmatrix}
\begin{bmatrix}
    z_{1, N} \\
 z_{2, N}
\\
 \vdots \\
 z_{N-1, N}\\
 z_{N, N} 
\end{bmatrix} 
\end{equation}
with initials 
 $z_{k, N}(0)=\exp(ikx_0)$ for $1\le k\le N$.
By Theorem \ref{maintheoremanalytic.thm1}, we  conclude that
\begin{equation}\label{CF.eq47}
|z_{1, N}(t)e^{-ix(t)}-1|\le 
C_0 N^{-3/2} e^{   R N t}
 \Big( \frac{\exp (|\Im {x}_0|+1)}{R } \Big)^{ (e-1)N/(2e-1)},  \ 0\le t\le T_{CF}^*,
\end{equation}
where
$R, C_0, T_{CF}^*$ are given by \eqref{maintheoremanalytic.thm1.eq1},  \eqref{maintheoremanalytic.thm1.eq3} and \eqref{maintheoremanalytic.thm1.eq4} 
respectively with $D_0(R)$ replaced by $R\ge 1$.
From the exponential convergence rate estimate in \eqref{CF.eq47}, we observe two key properties of the finite-section approximation \eqref{CF.eq46} that  its approximation error  is {\bf unaffected} by the real component of the initial condition $x_0$, and that the approximation error {\bf decreases} as the imaginary part of $x_0$ increases.
This suggests that the finite-section approximation \eqref{CF.eq46} possesses certain {\bf global} convergence characteristics, a significant advantage over the {\bf local} convergence characteristics of Carleman linearization shown in Theorem \ref{Carleman.thm},  and achieves rapid convergence under conditions where the dynamical system admits a satisfactory constant approximation.

\section{A Case Study}
\label{case.section}

In this section, we present a detailed  case study on the finite-section approximation \eqref{CF.eq46}
when the original dynamical system 
has the governing function  given in \eqref{simpleexample2.eq1}, and also compare the  performance of finite-section approximations to its
Carleman linearization 
and Carleman-Fourier linearization.

\smallskip
Let $z_{k,N}, 1\le k\le N$, be as in \eqref{CF.eq46}.
By induction on $k=N, N-1, \ldots, 1$, it can be verified that
\begin{equation} \label{simpleexample2.eq5-}
z_{k, N}(t)= e^{ik(at+x_0)} \Sp \sum_{l=0}^{N-k} ~ \frac{(k+l-1)!}{(k-1)! \Sp l!} \Sp \big(-e^{ix_0}(e^{iat}-1)\big)^l, \ \ 1\le k\le N.
\end{equation}
Therefore the first component  $z_{1, N}$ in the finite-section approximation \eqref{CF.eq46}
has the following explicit expression
 \begin{equation} \label{simpleexample2.eq5}
z_{1, N}(t)=e^{i(at+x_0)} \sum_{l=0}^{N-1} \big(-e^{ix_0}(e^{iat}-1)\big)^l,
\end{equation}
which is
 essentially  the Taylor polynomial of order $N-1$ for the exponential function $$e^{ix(t)}=e^{i(at+x_0)} \big(1+e^{ix_0}  (e^{iat}-1)\big)^{-1}$$ of the original state function $x(t)$ in \eqref{simpleexample2.eq2}.
As a  consequence of the explicit expression in
\eqref{simpleexample2.eq5}, for any initial state $x_0$ and in the time range $[0, T^*]$,  we have explicit approximation error 
\begin{equation} \label{simpleexample2.eq5+}
|z_{1,N}(t)e^{-ix(t)}-1|=|e^{ix_0}(e^{iat}-1)|^N= e^{|\Im x_0| N} |e^{iat}-1|^N,  \ \  N\ge 1.
\end{equation}

\begin{figure}[t]  
  \centering
   \includegraphics[width=6.9cm,  height=4.8cm]{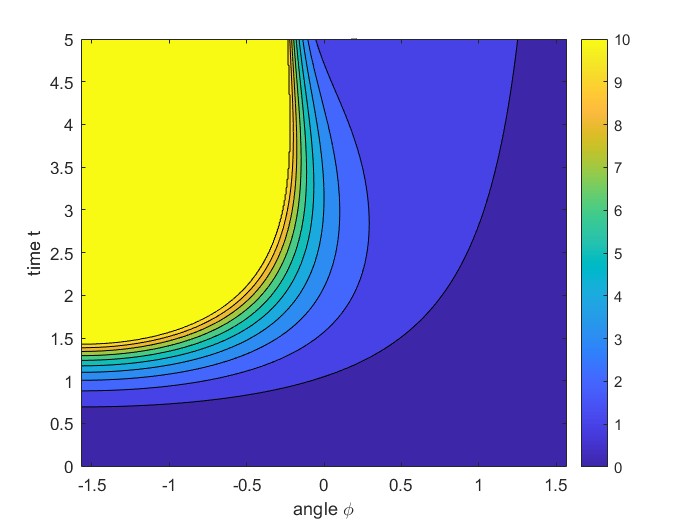}
   \includegraphics[width=6.9cm,  height=4.8cm]{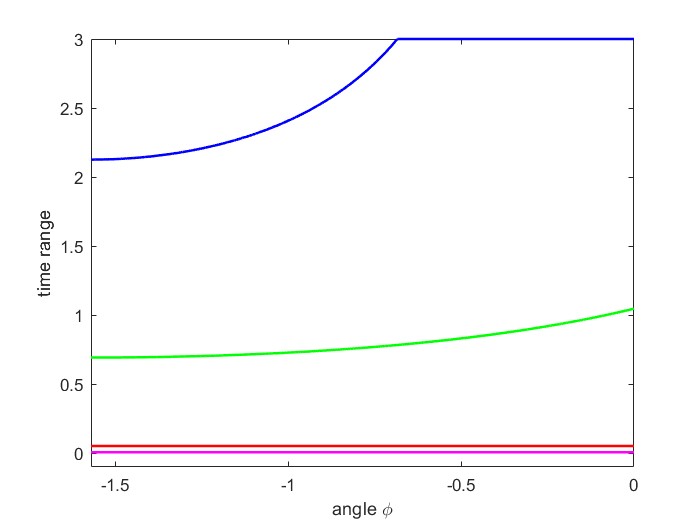}
        \captionsetup{width=1\linewidth}
\caption{Plotted on the left is  the function $\min\{ h(\varphi, t), 10\},  -\pi/2\le \varphi\le \pi/2, 0\le t\le 5$, where $h$ is given in \eqref{hh.def00}.
Presented on the right  is the actual time range $\min(T^*(\varphi), 3), -\pi/2/\le \varphi\le 0$, in \eqref{actualtimerage.orderone}, where $\Im x_0=0$ (in green) and
 $\Im x_0=2$ (in blue), and
the  time range $T_{CF}^*$ in Theorem  \ref{maintheoremanalytic.thm1} 
 when $\Im x_0=0$ (in red) and
 when $\Im x_0=2$ (in magenta).
}
  \label{fig:solu1}
\end{figure}

Write $a=e^{i\phi}$ as  \eqref{simpleexample2.eq2a} and define  
\begin{equation}
\label{hh.def00}
h(\phi, t)=|e^{iat}-1|^2= e^{-2t \sin \phi}-2 e^{-t\sin \phi} \cos(t\cos \phi)+1, \ t\ge 0,
\end{equation}
see the left plot of Figure \ref{fig:solu1} for the function $\min (h(\phi, t), 10), -\pi/2\le \phi\le \pi/2, 0\le t\le 5$.
 By \eqref{simpleexample2.eq5+},
  the first block, $z_{1, N}(t), N\ge 1$, in the finite-section approximation \eqref{CF.eq46} exhibits exponential convergence to $e^{ix(t)}$ in the time interval $[0, T^*)$ {\em if and only if}  the requirement
\begin{equation} \label{simpleexample2.eq6}
h(\phi, t) < \exp(2 \Im x_0) \quad \text{for all} \quad 0\le t <  T^*
\end{equation}
is satisfied, c.f.   \eqref{CF.eq47}.

In the case that  $\phi\in [-\pi/2, 0)$,  the function $ h(\phi, t), t\ge 0$ is unbounded. Therefore,  for any initial state $x_0$,
the actual time range
\begin{equation} \label{actualtimerage.orderone}
T^*(\phi)=\max \{T^*\ |\    \eqref{simpleexample2.eq6}\  {\rm   holds}\}
\end{equation}
for the convergence of $z_{1, N}(t), N\ge 1$, is finite. Illustrated
in  the right plot of Figure \ref{fig:solu1}
 are the maximal time range $\min (T^*(\phi), 5), -\pi/2\le \phi\le 0$, for $\Im x_0=0$ (in green) and $\Im x_0=2$ (in blue).
We remark that  the
 time range $T_{CF}^*$ in \eqref{maintheoremanalytic.thm1.eq4}, per  Theorem \ref{maintheoremanalytic.thm1}, is given by
\begin{eqnarray}\label{maximaltimerage.orderone}
T_{CF}^*& \hskip-0.08in = & \hskip-0.08in \sup_{\ln R>|\Im x_0|+1} \frac{e-1}{(2e-1) R} (\ln R-|\Im x_0|-1) = \frac{e-1}{(2e-1)}  e^{- |\Im x_0|-2}
\end{eqnarray}
see the right plot of Figure \ref{fig:solu1}, where
 $T_{CF}^*\approx 0.0524$ for $\Im x_0=0$ (in red) and
$T_{CF}^*\approx  0.0071$ 
for $\Im x_0=2$ (in magenta). 
We observe that  the time range $T_{CF}^*$ in \eqref{maintheoremanalytic.thm1.eq4} is independent on the selection of $a=\exp(i\phi)$, and it is much smaller than the actual time range $T^*(\phi), -\pi/2\le \phi<0$, for the exponential convergence of the finite-section approximation to the Carleman-Fourier  linearization.
Accordingly, the theoretical bound on
 the time range $T_{CF}^*$ in \eqref{maintheoremanalytic.thm1.eq4}
 should be considered as a conservative guarantee of exponential convergence for the finite-section approximation, rather than as a precise or sharp description of the actual time range.

For the scenarios when $\phi=0$, implying $a=1$ or equivalently $\Im a=0$, it can be verified that the maximum time range for the convergence of $z_{1, N}(t)$ can be evaluated explicitly,
\begin{equation*}
T^*(\phi)=\left\{\begin{array}{ll}
2 \arcsin \frac{\exp(\Im x_0)}{2} & {\rm if} \quad  \Im x_0\le \ln 2\\
+\infty & {\rm otherwise}.\end{array} \right.
\end{equation*}
Illustrated in the right plot of
 Figure \ref{fig:solu1} is
 $T^*(0)\approx \pi/3\approx 1.0472$ for $\Im x_0=0$ (in red) and
$T^*(0)=+\infty$ for $\Im x_0=2$ (in blue).

 For the case when $\phi\in (0, \pi/2]$, we have $0\le h(\phi,t)\le 4$.
 Using \eqref{simpleexample2.eq5}, we can conclude that  $z_{1, N}(t), N\ge 1$ in the finite-section approximation \eqref{CF.eq46} provides a satisfactory approximation to $e^{ix(t)}$ over the {\bf entire time range} $[0, \infty)$, provided that
 \begin{equation}\label{timerange.orderoneeq1} \Im x_0>\frac{1}{2} \ln h(\phi, t) \ {\rm for \ all} \ t\ge 0,\end{equation}
 which is the region above the green line on the right plot of Figure \ref{fig:solu1}, c.f. Remark \ref{global.remark}.

 Figures \ref{fig:cferrorOrderOne} depicts the approximation performance of the finite-section approach \eqref{CF.eq46}, where $a=e^{i\phi}$ and
\begin{eqnarray}
\label{fouriererror.def}
E_{CF}(x_0, T^*, N) &\hskip-0.08in  = & \hskip-0.08in  \max_{0\le t\le T^*} \log_{10}  \big| z_{1, N}(t) e^{-ix(t)}-1 \big|\nonumber\\
 & \hskip-0.08in  = & \hskip-0.08in
N \big(-\Im x_0\log_{10} e+ \frac{1}{2}\log_{10} \big(\max_{0\le t\le T^*}  h(\phi, t)\big)\big).
\end{eqnarray}
This demonstrates that the first component $z_{1, N}(t)$ in the finite-section approximation \eqref{CF.eq46} provides a better approximation to the original state $x(t)$ of the dynamical system \eqref{simpleexample2.eq1} in a longer time range when $\phi\in [-\pi/2, 0)$ and in the whole time range $[0, \infty)$ when $\phi\in (0, \pi/2]$, provided that  the imaginary  $\Im x_0$ of initial state $x_0$ takes larger value.
It is also observed that the proposed Carleman-Fourier linearization  has
better performance for the
complex dynamical system \eqref{simpleexample2.eq1} with the parameter $a$ having positive  imaginery part
than for  the one with the parameter $a$ having negative imaginery part. We believe that the possible reason
is that  the  finite-section approximation
\eqref{CF.eq46}  associated with the Carleman-Fourier linearization of the corresponding  dynamical system
 is stable when $\Im a<0$, while it is unstable when $\Im a>0$.

\begin{figure}[h] 
  \centering
    \includegraphics[width=4.9cm, height=3.8cm]{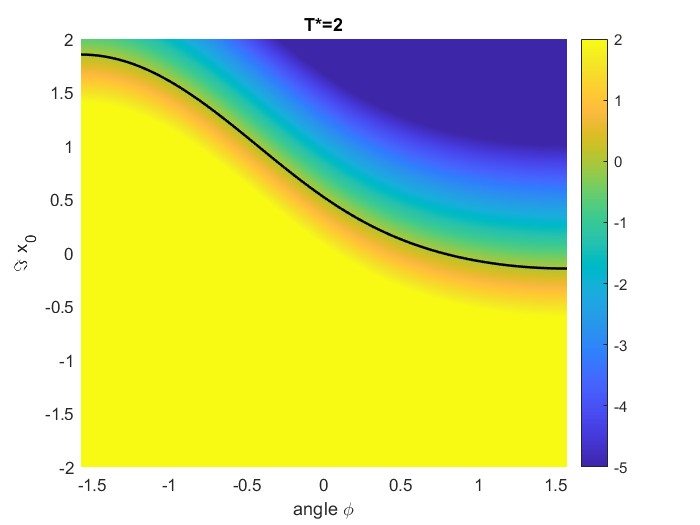}
    \includegraphics[width=4.9cm, height=3.8cm]{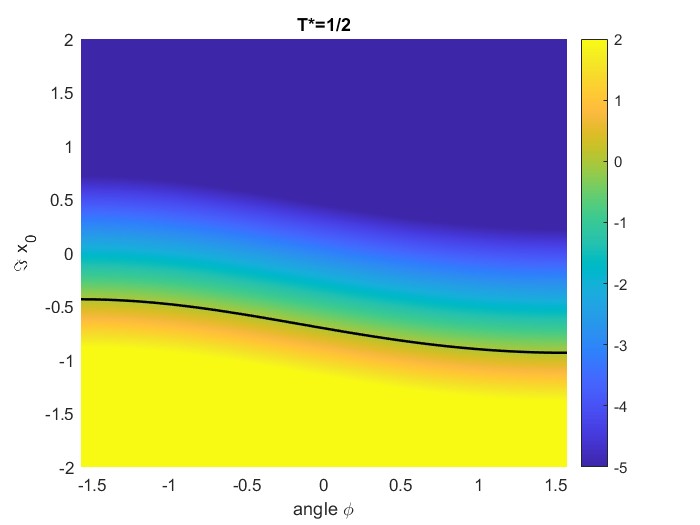}
    \includegraphics[width=4.9cm, height=3.8cm]{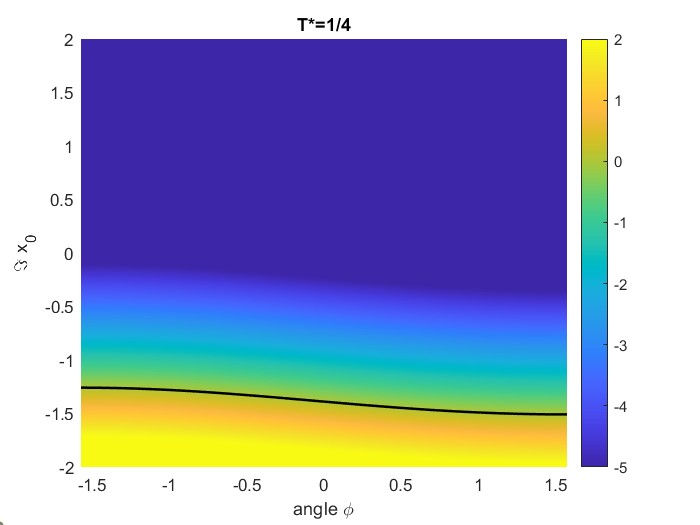}
    \\
\includegraphics[width=4.9cm, height=3.8cm]{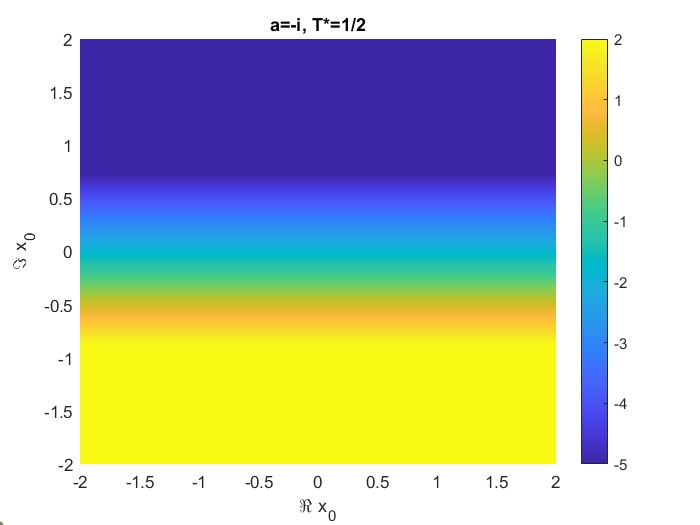}
    \includegraphics[width=4.9cm, height=3.8cm]{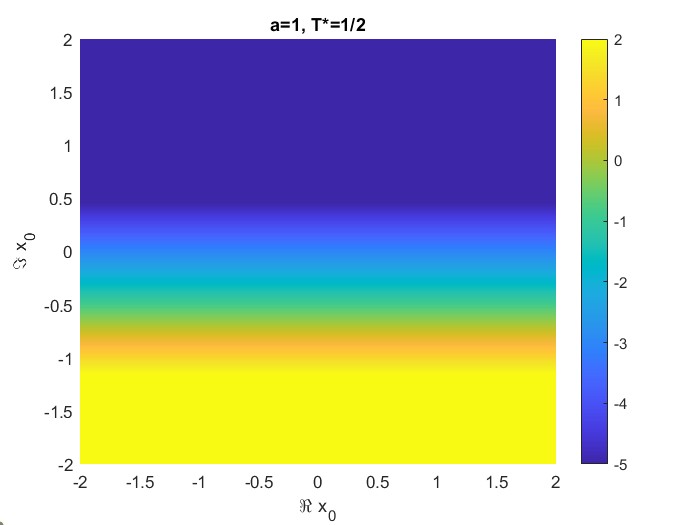}
\includegraphics[width=4.9cm, height=3.8cm]{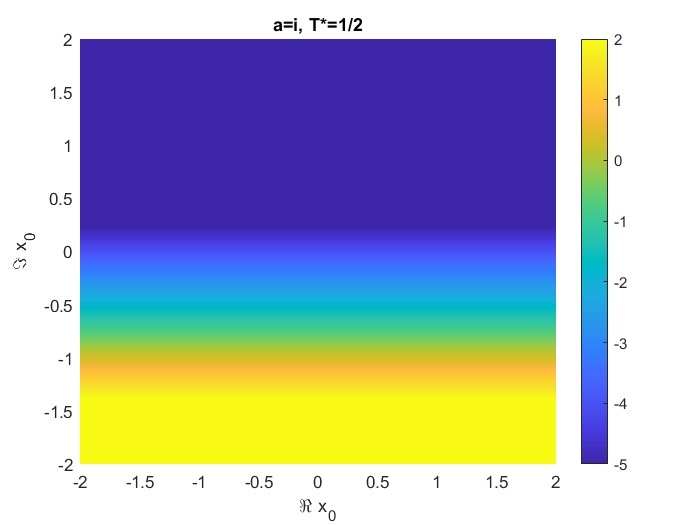}
         \caption{Plotted  on the top are
          the finite-section approximation errors $\max(\min(E_{CF}(x_0, T^*, N), 2), -5)$ with  $\-\pi/2\le\phi\le \pi/2$  as the $x$-axis
          and $-2\le \Im x_0\le 2$ as the $y$-axis,
          and level curve $E_{CF}(x_0, T^*, N)=0$ (in black)
          for $N=10$ and  $T^*=2$ (left),
           $1/2$ (middle) and $1/4$ (right) respectively. Shown on the bottom are
         $\max(\min(E_{CF}(x_0, T^*, N), 2), -5)$ with  $-2\le \Re x_0 \pi/2$  as the $x$-axis
          and $-2\le \Im x_0\le 2$ as the $y$-axis,
          for $N=10, T^*=1/2$ and  $\phi=-\pi/2$ (left),
           $0$ (middle) and $\pi/2$ (right) respectively. }
  \label{fig:cferrorOrderOne}
\end{figure}

\subsection{Comparison between  Carleman  and   Carleman-Fourier linearization }

Our numerical simulations show that the finite-section approximation of the Carleman-Fourier linearization
 exhibits  exponential convergence on the entire range if $\Im a\ge 0$ and  $\Im x_0>\ln 2$,
while   the finite-section approximation of  the Carleman linearization
has exponential convergence on the entire range when $\Im a<0$.
The possible reason  is that the dynamical system \eqref{simpleexample2.eq1} associated with the finite-section approximation of
the Carleman-Fourier linearization is stable when $\Im a>0$, while
the dynamical system  \eqref{simpleexample2.eq1} associated with  the finite-section approximation of the Carleman linearization is
stable when $\Im a<0$.

Comparing the performance between the Carleman-Fourier linearization and the Carleman linearization shown in Figures \ref{classicalcarleman}
and \ref{fig:cferrorOrderOne}, we see that the proposed Carleman-Fourier linearization has  much better performance than the Carleman linearization
 when $\Im x_0$ is large and the governing field is well approximated by trigonometric polynomials, while the  Carleman linearization, as expected,
  is a superior linearization technique of  a nonlinear dynamical system when the initial is not far from the origin.

\begin{appendix}

\section{A dynamical system with periodic governing function}
\label{example.section}

In this appendix, we consider the trajectory behaviour of
 the  dynamical system \eqref{simpleexample2.eq1}, which 
 could blow up at a finite time, exhibit a limit cycle, and converge or diverge, see
 Figure \ref{fig:solution1}.

By \eqref{simpleexample2.eq2}, the solution $x(t)$ of the complex dynamical system \eqref{simpleexample2.eq1} may blow up at a finite time $t=t_0>0$ if the initial $x_0$ satisfies
\begin{equation}\label{simpleexample2.eq2b}
    1 +(e^{ait_0} - 1)e^{ix_0} = 0 \quad \text{and} \quad 1 +(e^{ait} - 1)e^{ix_0} \neq 0 \  \ {\rm for \ all} \ \ 0\le t< t_0,
\end{equation}
 see the black trajectories shown in Figure \ref{fig:solution1}
where simulation parameters $x_0=i\ln  (1-e^{ai\pi/2})$ for $a=1,  i, -i$ respectively.
 One may verify that the requirement \eqref{simpleexample2.eq2b} for the initial state vector $x_0$ is met for some $t_0 > 0$ when $\Re(e^{ix_0}) = \frac{1}{2}$ and $a= 1$, or when $\Re x_0\in 2\pi {\mathbb Z}+\pi$ and $a=-i$, or when $\Re x_0\in 2\pi {\mathbb Z}$ and $\Im x_0<0$ and $a=i$.

\begin{figure}[h] 
  \centering
    \includegraphics[width=4.9cm, height=4.5cm]{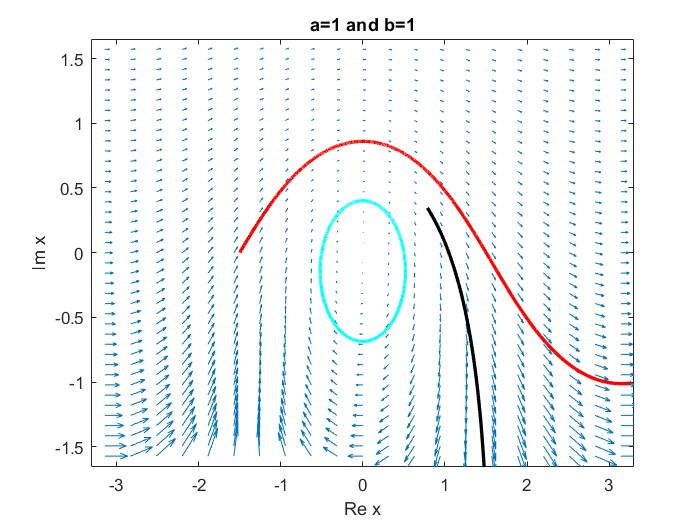}
       \includegraphics[width=4.9cm, height=4.5cm]{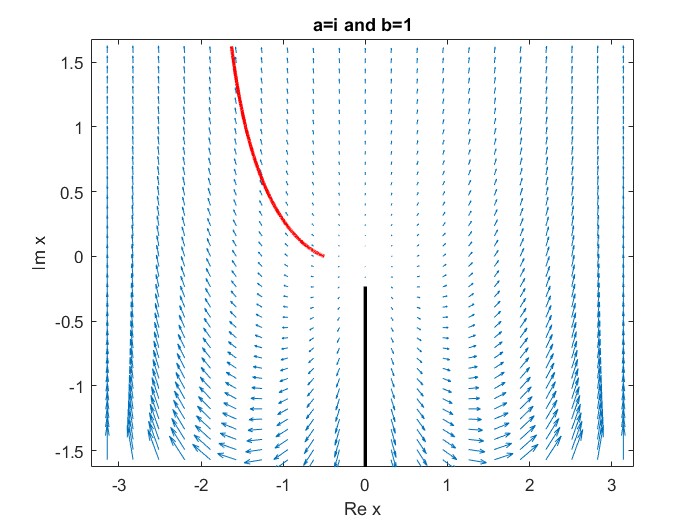}
       \includegraphics[width=4.9cm, height=4.5cm]{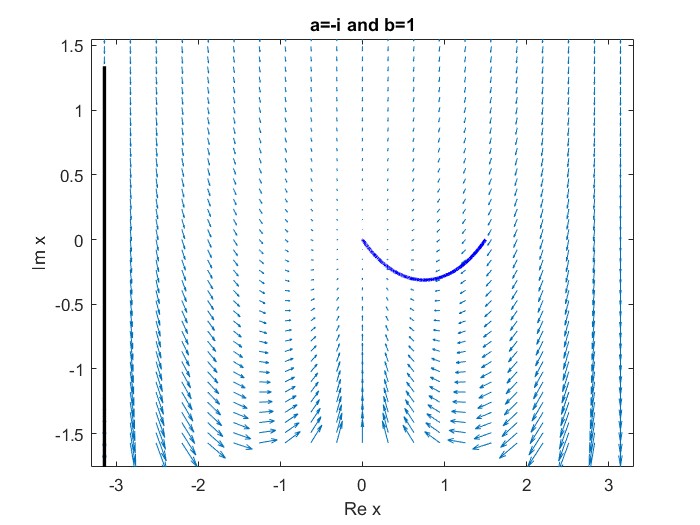}
        \captionsetup{width=1\linewidth}
\caption{Plotted are the vector fields $ a(1-e^{ix})$ of the complex dynamical system \eqref{simpleexample2.eq1} with  $a=1$ (left),  
$a=i$ (middle) and $a=-i$ (right), where $-\pi\le \Re x\le \pi$ and $-\pi/2\le \Im x\le \pi/2$. Trajectories
on the left figure has parameter $a=1$ (left) 
   and initial
 $x_0= i\ln  (1-e^{ai\pi/2})\approx 0.7854 a + 0.3466i$ (in black), $-1/2a$ (in cyan) and $-3/2a$ (in  red).
 Presented in the middle is trajectories with $a=i$ and $x_0= i\ln  (1-e^{ai\pi/2}) \approx - 0.2330i$ (in black) and $-1/2$ (in red),
 while on the right are trajectories with $a=-i$ and $x_0= i\ln  (1-e^{ai\pi/2}) \approx -3.1416 + 1.3378i$ (in black) and $3/2$ (in blue).
 Trajectories shown in the figures may
 blow up in a finite time (in black), have limit cycle (in cyan), converge (in  blue) and  diverge (in red).
}
  \label{fig:solution1}
\end{figure}

Now we continue examining the behavior of the dynamical system \eqref{simpleexample2.eq1} when the initial vector $x_0$ does not satisfy condition \eqref{simpleexample2.eq2b} for all $t_0> 0$, i.e., $1 +(e^{ait} - 1)e^{ix_0}\ne 0$ for all $t\ge 0$.
For the case that $\Im a=0$, i.e., $a=1$.
 we observe that $e^{-ix_0}(1+(e^{it}-1)e^{ix_0}), t\ge 0$, is a  circle with center  $e^{-ix_0}-1$ and radius 1.
Therefore  $\ln \left(1-(e^{it}-1)e^{ix_0}\right)$ is a periodic function with a period of $2\pi$ when $|e^{-ix_0}-1|>1$, and
$\ln \left(1-b(e^{it}-1)\right)-it$ is a periodic function with the same period of $2\pi$ when  $|e^{-ix_0}-1|<1$. This implies that when $a=1$, the dynamical system \eqref{simpleexample2.eq1} diverges when $|e^{-ix_0}-1|<1$ and exhibits a limit cycle when $|e^{-ix_0}-1|>1$.
 These behaviors are illustrated by the cyan color limit cycle trajectory in Figure \ref{fig:solution1} with a period of $2\pi$ and the red color trajectory in Figure \ref{fig:solution1}, where $x(t)-t$ forms a periodic function with a period of $2\pi$.

For the case that $\Im a\ne 0$, we observe that
(i) $\lim_{t\to \infty} 1+e^{ix_0}(e^{ait}-1)=1-e^{ix_0}$ when $\Im a>0$; and (ii)
  $\lim_{t\to \infty} e^{-iat} \left(1+e^{ix_0}(e^{ait}-1)\right)=e^{ix_0}$ when $\Im a<0$.
Therefore, the dynamical system \eqref{simpleexample2.eq1} converges when $\Im a<0$, diverges  when $\Im a>0$ and $x_0\not\in 2\pi {\mathbb Z}$, and the solution of the dynamical system \eqref{simpleexample2.eq1} remains at the equilibrium $x_0\in 2\pi {\mathbb Z}$.
 This behavior is illustrated by the green color trajectory in the right plot of  Figure \ref{fig:solution1},  
 and the  red color trajectory in the middle plot of Figure \ref{fig:solution1}. 

\section{Proof of Theorem \ref{Carleman.thm} }
\label{proof.section}

In this appendix, we adopt the procedure outlined in \cite{Amini2022} to establish Theorem \ref{Carleman.thm},  under modified assumptions on the Taylor coefficients of the governing function.
First, we need a local bound estimate for the state vector $x$
of the dynamical system \eqref{Carleman.thm.eq0}.

\begin{lemma} \label{boundestimate.lem} For any given $M_0>|x_0|$, we have
\begin{equation} \label{proof.eq1}
|x(t)|\le M_0 \ \  {\rm for \ all}\  \  0\le t\le \frac{M_0R_0}{C_0(e^{R_0M_0}-1)} \ln \frac{M_0}{|x_0|}, 
\end{equation}
 where  $C_0$ and $R_0$ are constants in \eqref{Carleman.thm.eq1}.
\end{lemma}

\begin{proof}
By the continuity of the state vector $x$, either $|x(t)|\le M_0$ for all $t\ge 0$ or there exists some $T>0$ such that
$|x(t)|\le M_0$ for all $0\le t\le T$ and $|x(T)|=M_0$.
Clearly it suffices to prove 
\begin{equation} \label{proof.eq1+}
    T\ge \frac{M_0R_0}{C_0(e^{R_0M_0}-1)} \ln \frac{M_0}{|x_0|}
\end{equation}
for the second case. 
Observe that
\begin{eqnarray*}
    |x(t)|&\hskip-0.08in \leq & \hskip-0.08in |x_0|+\int_{0}^t|f(s)|ds
    \le |x_0|+\sum_{n=1}^{\infty}|c_n| \int_{0}^t |x(s)|^nds\\
    & \hskip-0.08in \leq  &  \hskip-0.08in |x_0|+ \sum_{n=1}^{\infty}\frac{C_0R_0^{n-1}}{n!}M_0^{n-1}
    \int_{0}^t |x(s)| ds\\
    & \hskip-0.08in = & \hskip-0.08in |x_0|+\frac{C_0(e^{M_0 R_0}-1)}{M_0R_0}    \int_{0}^t |x(s)| ds \ \ {\rm for \ all}    \ 0\le t\le T, \end{eqnarray*}
which implies that
$$ \exp\Big(- \frac{C_0(e^{M_0R_0}-1)}{M_0 R_0} t\Big)  
\Big(\int_0^t |x(s)|ds+  \frac{M_0R_0|x_0|}{C_0(e^{M_0R_0}-1)}\Big)
$$
is a decreasing function on $[0, T]$. Hence
$$|x(t)|\le |x_0|  \exp\Big(\frac{C_0(e^{M_0R_0}-1)}{M_0R_0} t\Big) \ \ {\rm for  \ all }\ \  0\le t\le T. $$
This together with  the assumption $|x(T)|=M_0$
of the second case proves \eqref{proof.eq1+} and completes the proof.
\end{proof}

Now we start the proof of Theorem \ref{Carleman.thm}.

\begin{proof}[Proof of Theorem \ref{Carleman.thm}]
Let $x_{k, N}, 1\le k\le N$, be as in \eqref{simpleexample2.eq7}  and set
$y_k(t)=(x(t))^k-x_{k,N}(t), 1\le k\le N$. Then, $y_k, 1\le k\le N$, satisfy the following ODEs,
\begin{equation*}
    \frac{dy_k}{dt}=kc_1y_k+\sum_{n=k+1}^Nkc_{n-k+1}y_n+\sum_{n=N+1}^{\infty}kc_{n-k+1}x^n
\end{equation*}
with initials $y_k(0)=0$. Multiplying $\exp(-c_1kt)$ 
at both sides 
of the above ODEs 
and then integrating, we obtain
\begin{eqnarray*}
    y_k(t)=k\int_{0}^te^{c_1k(t-s)}\Bigg[\sum_{n=k+1}^Nc_{n-k+1}y_n(s)+\sum_{n=N+1}^{\infty}c_{n-k+1}(x(s))^n\Bigg]ds.
\end{eqnarray*}

Take $M_0>|x_0|$ and set $T= \frac{M_0R_0}{C_0(e^{M_0R_0}-1)} \ln \frac{M_0}{|x_0|}$.
Therefore for all $0\le t\le T$ and $1\le k\le N$, we obtain from 
\eqref{Carleman.thm.eq1} and Lemma \ref{boundestimate.lem} 
that 
\begin{eqnarray}\label{proof.eq2}
    |y_k(t)|
    &\hskip-0.08in \leq &  \hskip-0.08in C_0 k \int_{0}^te^{C_0k  (t-s)}\Bigg[\sum_{n=k+1}^N\frac{R_0^{n-k}}{(n-k+1)!}|y_n(s)|+\sum_{n=N+1}^{\infty}\frac{M_0^n R_0^{n-k }}{(n-k+1)!}\Bigg]ds\nonumber\\
    & \hskip-0.08in \leq& \hskip-0.08in 
       C_0 k  \int_{0}^te^{C_0k(t-s)}  
        \Bigg[  \sum_{n=k+1}^N\frac{R_0^{n-k}}{(n-k+1)!} |y_n(s)|+
     \frac{(e^{M_0R_0}-1) M_0^{N} R_0^{N-k}}{  (N+1-k)!}
    \Bigg] ds,
\end{eqnarray}
where the last estimate holds as
\begin{eqnarray*}
& & \sum_{n=N+1}^\infty \frac{M_0^n R_0^{n-k} }{(n-k+1)!} =  \sum_{m=0}^\infty \frac{M_0^{m+N+1} R_0^{m-k+N+1} }{(N+1+m-k+1)!}\\
& \le & \frac{M_0^{N} R_0^{N-k}}{(N+1-k)!}
\sum_{m=0}^\infty \frac{M_0^{m+1} R_0^{m+1} }{(m+1)!}
 = \frac{(e^{M_0R_0}-1) M_0^{N} R_0^{N-k}}{  (N+1-k)!}.
\end{eqnarray*}

Let $\tilde R_0=\max(1,  R_0|x_0|e^2)$ as in  \eqref{tildeR0.def}  and set 
$$u_k(t)= (e^{M_0R_0}-1)^{-1} M_0^{-N} (R_0/\tilde R_0)^{k-N}  |y_k(t)|,\  1\le k\le N.$$
Then, for $1\le k\le N$, we obtain from 
\eqref{proof.eq2} that
\begin{eqnarray}  \label{proof.eq3}
 u_k(t) & \hskip-0.08in \le & \hskip-0.08in   
  C_0 k K  \int_{0}^te^{C_0k(t-s)}  
        \Big[  \sum_{n=k+1}^N  u_n(s)+1\Big]
      ds\nonumber\\
       & \hskip-0.08in \le & \hskip-0.08in  
        C_0 k K  \int_{0}^te^{C_0k K (t-s)}  
        \Big[  \sum_{n=k+1}^N  u_n(s)+1\Big]
      ds,
\end{eqnarray}
where 
\begin{equation} \label{proof.eq4}
1\le K:=\sup_{m\ge 0} \frac{\tilde R_0^m}{(m+1)!}\le 
\sum_{m= 0}^\infty \frac{\tilde R_0^m}{(m+1)!}=
\frac{e^{\tilde R_0}-1}{ \tilde R_0}\le  
e^{\tilde R_0}/\tilde R_0.\end{equation}

Taking $k=N$ in \eqref{proof.eq3} gives
\begin{equation*} 
u_{N}(t)+1\le       e^{C_0KNt}, \ \ 0\le t\le T. 
\end{equation*}
By induction on $k=N, N-1, \ldots, 2$, we can show that
\begin{equation} \label{proof.eq5}\sum_{n=k}^N u_n(t)+1\le \frac{N^{N-k}}{(N-k)!} e^{C_0KNt}, \ 0\le t\le T.
\end{equation}
Substituting the above estimate with $k=2$ into the right hand side of the estimate
\eqref{proof.eq3} with $k=1$ gives
\begin{eqnarray}  \label{proof.eq6}
 u_1(t) 
& \hskip-0.08in \le & \hskip-0.08in          C_0  K  \frac{N^{N-2}}{(N-2)!} 
        \int_{0}^t e^{C_0  K (t-s)}  e^{C_0KNs}
      ds\nonumber\\
& \hskip-0.08in \le & \hskip-0.08in  \frac{N^{N-2}}{(N-1)!} e^{C_0KNt}\le (2\pi)^{-1/2} N^{-3/2} e^{N+C_0KNt},\ \ 0 \le t\le T,
      \end{eqnarray}
    where we use the Stirling inequality $N!> \sqrt{2\pi} N^{N+1/2} e^{-N}$.
    Therefore 
    \begin{equation} \label{proof.eq7}
    |y_1(t)|
    \le  \frac{(e^{M_0R_0}-1)\tilde R_0}{\sqrt{2\pi} R_0 }  N^{-3/2} 
   \Big (\frac{M_0R_0 e}{\tilde R_0} e^{C_0 e^{\tilde R_0} t/\tilde R_0 } \Big)^N.
    \end{equation}

   Let $T^*$ be as in \eqref{Tstar.def},  fix     $0<t< T^*$, and set
    $M_0=|x_0| \exp(C_0 e^{\tilde R_0} t)$.
    Then 
    $M_0 R_0\le R_0|x_0| e^2 \le \tilde R_0$
    and 
     $$ \frac{M_0R_0}{C_0(e^{M_0R_0}-1)} \ln \frac{M_0}{|x_0|}\ge \frac{M_0R_0 e^{\tilde R_0}}{e^{M_0R_0}-1}  t\ge t.$$
Therefore
$$|x(s)|\le M_0 \ \ {\rm for \ all} \ \ 0\le s\le t$$
by Lemma \ref{boundestimate.lem}.
Then applying \eqref{proof.eq7} with the above selection of $M_0$ yields
$$|y_1(s)|\le \frac{ \tilde R_0 e^{\tilde R_0}}{\sqrt{2\pi} R_0} N^{-3/2}
\Big(\frac{ R_0|x_0| e}{\tilde R_0} e^{C_0 e^{{\tilde R}_0} (t+ s /\tilde R_0) } \Big)^N \ \ {\rm for \ all}\ \ 0\le s\le t.$$
The desired estimate \eqref{Carleman.thm.eq2} follows from the above inequality by taking $s=t$.
 \end{proof}

\end{appendix}

\end{document}